\documentclass{birkjour_t2}
 \newtheorem{thm}{Theorem}[section]
 
 \newtheorem{lem}[thm]{Lemma}
 \newtheorem{prop}[thm]{Proposition}
 \theoremstyle{definition}
 \newtheorem{defn}[thm]{Definition}
 \theoremstyle{remark}

 \numberwithin{equation}{section}

\begin{document}

%
%
%
%
%
%
%
%
%

\title[Lower bounds for A-singularities]{Some lower bounds for the maximal number of A-singularities in algebraic surfaces. II}

\author[J.G.Escudero]{Juan Garc\'{\i}a Escudero}
\address{%
Madrid, Spain
\\
https://orcid.org/0000-0002-7422-7908
}
\email{jjgemplubg@gmail.com}

\keywords{singularities, algebraic surfaces}

\subjclass{14B05, 14J17, 14J70}

\date{April,  2024}
\dedicatory{}

\begin{abstract}
Algebraic surfaces in the complex  projective space  with a large number of $A$-type singularities have been presented in a recent paper. We extend the construction in order to obtain new lower bounds for the maximal number of  $A$ singularities in degree $d$ complex surfaces.

\end{abstract}
\maketitle

\section{Introduction}
\bigskip\par

A classical problem in algebraic geometry consists of determining the possible number of singularities of a given type for surfaces of degree $d$ in the complex  projective space  $\mathbb{P}^{3}({\Bbb{C}})$. Upper bounds for the maximal number of $A_{\nu}$-type singularities on degree $d$ complex algebraic surfaces have been given in \cite{var83, miy84}. Lower bounds for degree $d$ have been obtained in \cite{chm92, esc13} for $\nu = 1$ ( nodes ) and  in \cite{lab06, esc14b, esc26}  for $\nu > 1$. For a number of low degree surfaces certain explicit constructions give better lower bounds for nodes \cite{ cay69, kum64, tog40, bar96, end97, sar01, lab06b} and also for $\nu > 1$ \cite{lab06b, bor13, bon25}. The exact value of the upper bounds is known only for a few cases \cite{bea82, jaf97}.
 \par
Complex algebraic surfaces of degrees $d=3q$  with a large number of type $A$ singularities have been constructed in a recent paper  \cite{esc26}. We extend the construction to obtain new lower bounds for the maximal number of $A_{\nu}$ singularities for $\nu > 2$.
\par
The affine equations of the surfaces studied in \cite{esc26} consist of a sum of a bivariate polynomial ${\mathcal{J}}$ with three critical values and a univariate polynomial ${\mathcal{B}}$ with two critical values.  The surfaces constructed in this paper have similar equations with new polynomials ${\mathcal{G}}$ replacing ${\mathcal{B}}$. The polynomials ${\mathcal{G}}$ also have two critical values and can be described in terms of planar trees which are obtained by a series of transformations applied to the trees associated to ${\mathcal{B}}$ and others. In this way the results about cusps and other $A_{\nu}$ singularities obtained in \cite{esc14a, esc14b} can be seen as special cases of this constructions. 

 \section{The families of polynomials with two critical values}
 \bigskip\par
 A univariate polynomial with no more than two different critical values, also called a Belyi polynomial, can be represented by a plane tree with a bicoloring for the vertices. The polynomial critical points have the multiplicities given by the number of edges adjacent to the vertices minus one \cite{adr98, lab06}. The degree of a vertex is the number of edges incident to it. A leaf vertex is a vertex with degree one. Black and white vertices represent critical points with critical value $\zeta=-1$ and $\zeta=1$ respectively, and their multiplicities are given by the degrees of the vertices minus one. We denote by $N_{\zeta}(\mathcal{P}, \rho)$ the number of critical points of a polynomial $\mathcal{P}$ with critical value $\zeta$ and multiplicity $\rho$.
 \par
 The polynomials ${\mathcal{B}}^{(t)}_{d,\nu,\epsilon}(w)$, $t=1, 2, 3$, obtained in \cite{esc26}, Lemmas 1 (Eq. 5), 2 (Eq. 8) and 3 (Eq. 9) are
 \begin{equation}  
{\mathcal{B}}^{(1)}_{3(n+3m(n+m)), 3(n+m)-1}(w)
\end{equation}
\par\noindent
with $n\in {\Bbb{Z}}^{\geq 0},  m\in {\Bbb{Z}}^{+}$,
\begin{equation} 
{\mathcal{B}}^{(2)}_{3(m+lj+(n+l)(3l+j+1))+j(j+2), 3(n+l)+j, 3m+j-1}(w)
\end{equation} 
\par\noindent
with $j\in\{0,1\}, n, m\in {\Bbb{Z}}^{+}$ if $j=0$, $n,m \in {\Bbb{Z}}^{\geq 0}$ if $j=1, l=m, m+1, m+2, \dots $, and 
\begin{equation} 
{\mathcal{B}}^{(3)}_{3(m(3(m+n)+x-1)+j(n+2m+\lfloor \frac{x}{3}\rfloor)+l+1), 3(n+m)+j+x-1, 3l+2\lfloor 1+(\lfloor\frac{x-1}{2}\rfloor-\frac{1}{2})j\rfloor+j\lfloor \frac{1}{x}\rfloor}(w)  
\end{equation} 
\par\noindent
with $j\in\{-1,0,1\}, x\in\{1,2,3\},  m\in {\Bbb{Z}}^{+}, n\in {\Bbb{Z}}^{\geq 0}, 3m+j \geq 4$; $l=0$  if $m=1$ and $l=0,1,\dots ,m-2$ if $m\geq 2$. 
 \bigskip\par
As discussed in \cite{esc26}, if we compare the trees associated with the polynomials with given values of $d,\nu,\epsilon$ we can see that some polynomials in Lemma 3 coincide with some in Lemmas 1 and 2. With the notation ${\mathcal{B}}^{(1)}[n, m]$, ${\mathcal{B}}^{(2)}[j, n, m, l]$ and ${\mathcal{B}}^{(3)}[x, p, n, m, l]$ ($p=3m+j$) we showed: 
 \bigskip\par\noindent
1) ${\mathcal{B}}^{(1)}[n+1, m]={\mathcal{B}}^{(3)}[2, 3m+1,n,m,0]$ with $n\in {\Bbb{Z}}^{\geq 0}$ and $m\in {\Bbb{Z}}^{+}$. 
 \bigskip\par\noindent
 2) ${\mathcal{B}}^{(2)}[0, n+1, l+1, m]={\mathcal{B}}^{(3)}[3, 3m+1,n,m,l]$ for $n\in {\Bbb{Z}}^{\geq 0}$,  $m\in {\Bbb{Z}}^{\geq 1}$, $l=0$ if $m=1$, $l=0,1, \dots m-2$ if $m \geq 2$.
\bigskip\par\noindent
 3) ${\mathcal{B}}^{(2)}[1, n+1, l, m-1]={\mathcal{B}}^{(3)}[3, 3m-1, n, m, l]$ with $n\in {\Bbb{Z}}^{\geq 0}$, $m\in {\Bbb{Z}}^{\geq 2}$, $l=0,1, \dots m-2$. 
 \bigskip\par
 We use the trees of ${\mathcal{B}}_{d_{0},\nu,\epsilon}^{(t)}$ as starting trees of a series of tree sequences. 
We denote the corresponding polynomials by ${\mathcal{G}}^{(t)}_{d,\nu,\epsilon}(w), t=1,2, 3$ with ${\mathcal{G}}^{(t)}_{d_{0},\nu,\epsilon}(w)={\mathcal{B}}_{d_{0},\nu,\epsilon}^{(t)}(w)$. The label $\epsilon$, used for $t=2,3$ in ${\mathcal{G}}^{(t)}$ when it is not zero, refers only to the initial tree.
\par
     \begin{defn} 
     The triple $(d,N_{-1}(\mathcal{G}, \nu), N_{1}(\mathcal{G}, \nu))$ is said to satisfy condition $(E)$  if $ \Big \lfloor \frac{d}{\nu+1} \Big \rfloor=N_{-1}(\mathcal{G}, \nu)$ and   $\Big \lfloor \frac{d-1}{\nu}   \Big \rfloor-    \Big \lfloor \frac{d}{\nu+1}   \Big \rfloor=N_{1}(\mathcal{G}, \nu)$.
  
      \end{defn} 
 The polynomials ${\mathcal{B}}^{(t)}_{d,\nu,\epsilon}(w)$, $t=1, 2, 3$, satisfy condition $(E)$ \cite{esc26}.  In order to describe the generation of ${\mathcal{G}}^{(t)}$, we employ a formal language $\mathcal{L}$ consisting of an alphabet $\Sigma=\{\alpha, \beta , \dots  \} $ and a set of admissible words. We denote the tree by $\mathcal{T}_{\nu}$ if $\nu$ is the largest value of the multiplicity of the critical points with critical value $-1$. We consider the sequence of transformations 
   \bigskip\par
       \begin{equation} 
\mathcal{T}_{\nu,0}  { w_{1} \atop \longrightarrow } \mathcal{T}_{\nu,1}  { w_{2} \atop \longrightarrow } \mathcal{T}_{\nu,2} { w_{3} \atop \longrightarrow } \dots { w_{n} \atop \longrightarrow } \mathcal{T}_{\nu,n}  
    \end{equation} 
    where the initial tree $\mathcal{T}_{\nu,0}$ is the tree of ${\mathcal{B}}_{d_{0},\nu,\epsilon}^{(t)}$ and we associate the word $w=w_{1}w_{2} \dots w_{j}$ to  $\mathcal{T}_{\nu,j}$. 
\par\noindent
 
     \begin{defn}  A word $w=w_{1}w_{2} \dots w_{n}, w_{j} \in \Sigma, j=1, 2, \dots n$ with length $|w|=n$ is said to be $E$-admissible if all the trees $\mathcal{T}_{\nu,j}$ obtained in Eq. (2.4) give polynomials $\mathcal{G}$ with  $(d, N_{-1}(\mathcal{G}, \nu), N_{1}(\mathcal{G}, \nu))$ satisfying condition $(E)$.  The language $\mathcal{L}_{E}$  is the set of $E$-admissible words $w$.
       \end{defn} 
       
                \begin{figure}[h]
 \includegraphics[width=16pc]{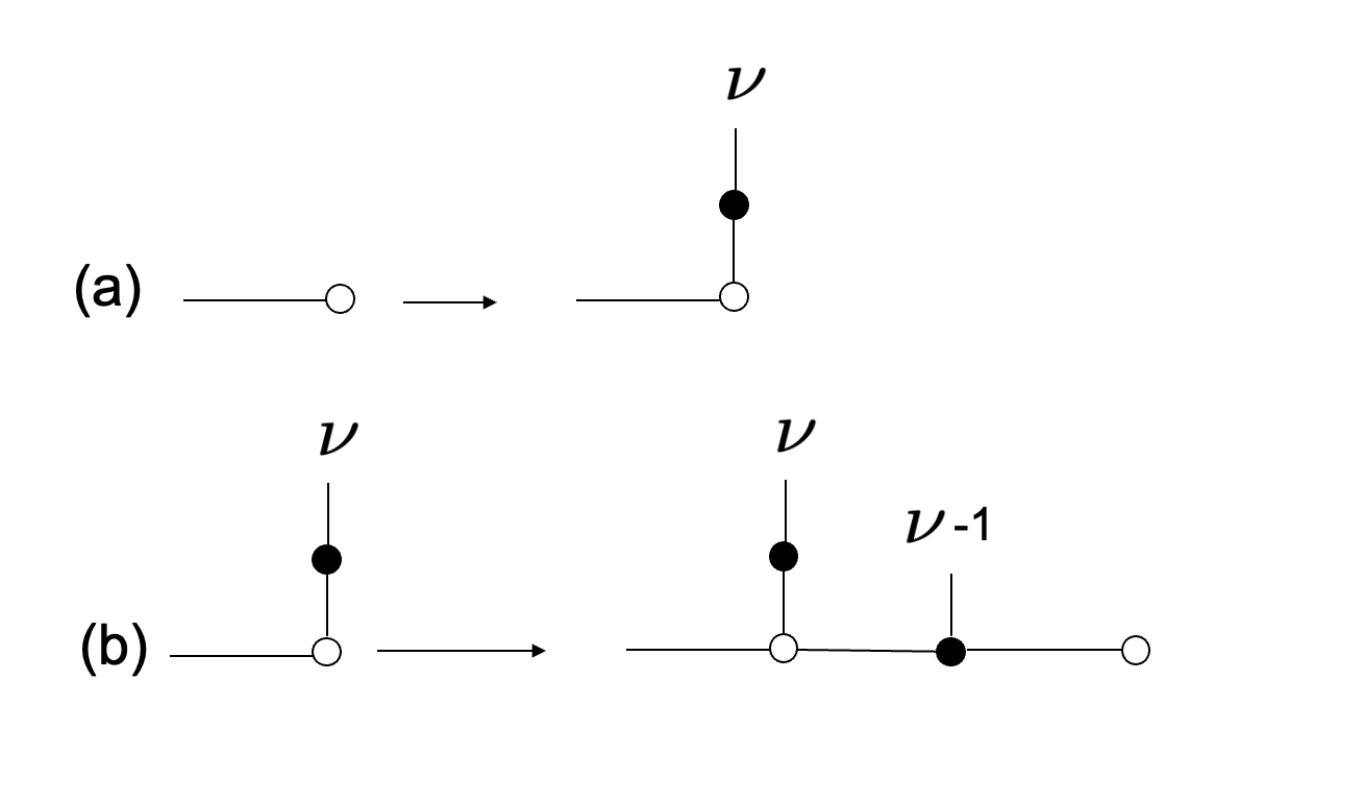}
\caption{\label{label} The transformations or substitution rules for the derivation of ${\mathcal{G}}^{(t)}, t=1,3$: (a) $\alpha$, where we only show one edge with a number, in this case represented by $\nu$, indicating the number of additional adjacent edges, b) $\beta$.}
\end{figure}  
\par
          \begin{figure}[h]
 \includegraphics[width=20pc]{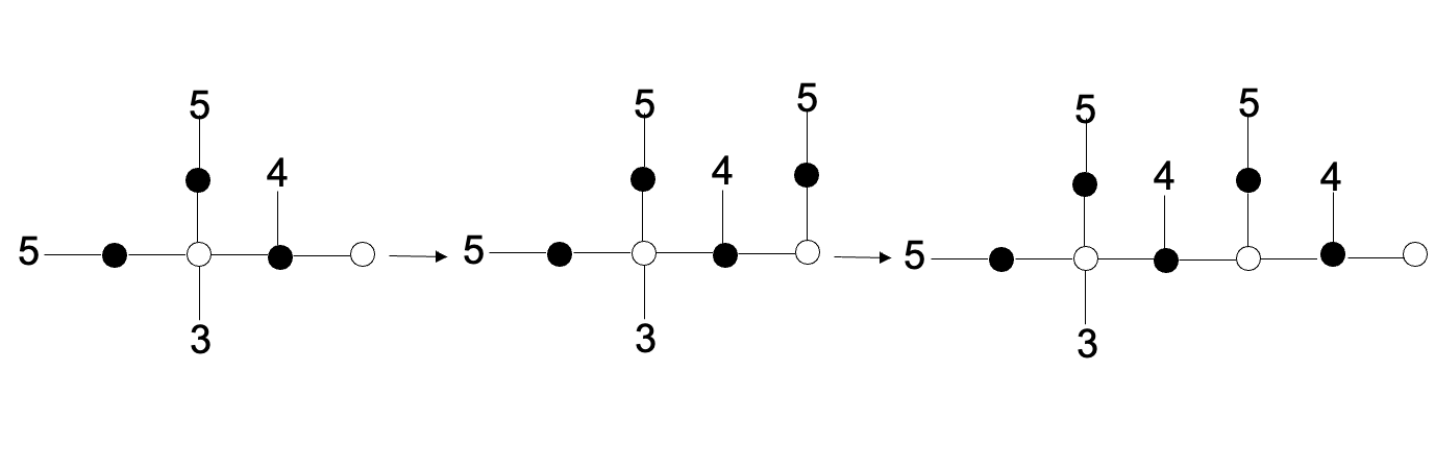}
\caption{\label{label} Applying  the substitution rules to the tree corresponding to ${\mathcal{G}}^{(1)}_{d_{0}, 5}(w)={\mathcal{B}}^{(1)}_{21, 5}(w)$. In the first step we apply $\alpha$ which produces the tree of ${\mathcal{G}}^{(1)}_{27,5}(w)$, and in the second we apply $\beta$ to get the ${\mathcal{G}}^{(1)}_{33,5}(w)$ tree.}
\end{figure}  
 
\par
    
We consider several types of alphabets and word sequences. Some words may represent trees which correspond to the same polynomial.  Although the values of the degrees $d$ depend on $d_{0}$,  in the notation $d_{0}$ is suppressed.   
 \par  
For the cases ${\mathcal{B}}^{(t)}, t=1,3$, the alphabet is  $\Sigma=\{\alpha, \beta\}$, where the transformations $\alpha$ and $\beta$  are depicted in Fig.1. The substitution rule $\alpha$ adds to a leaf white vertex one black vertex with $\nu$ adjacent white vertices, and $\beta$ is a similar transformation applied to a non leaf white vertex:
\par\noindent   
   \hspace{2cm}   $\alpha: \mathcal{T}_{\nu}(d, N_{-1}, N_{1}) \longrightarrow \mathcal{T}_{\nu}(d+\nu+1, N_{-1}+1, N_{1})$, 
\par\noindent
   \hspace{2cm}   $\beta: \mathcal{T}_{\nu}(d, N_{-1},N_{1})  \longrightarrow \mathcal{T}_{\nu}(d+\nu+1, N_{-1}+1, N_{1}+\Big \lfloor \frac{2}{\nu} \Big \rfloor )$. 
\par
   In Fig.2 the transformations $\alpha$ and $\beta$ are applied succesively to the tree associated to  ${\mathcal{B}}^{(1)}_{21, 5}(w)$ and we get 
   
         \begin{equation} 
{\mathcal{G}}^{(1)}_{21, 5}(w) ={\mathcal{B}}^{(1)}_{21, 5}(w)  { \alpha \atop \longrightarrow } {\mathcal{G}}^{(1)}_{27,5}(w)   { \beta \atop \longrightarrow } {\mathcal{G}}^{(1)}_{33,5}(w).
    \end{equation} 
    \par
    \begin{lem}     There exist polynomials  ${\mathcal{G}}^{(t)}_{d,\nu,\epsilon}(w), t=1, 3$ whose trees are obtained from initial trees of type ${\mathcal{B}}_{d,\nu,\epsilon}^{(t)}, t=1,3$ respectively, by means of a series of transformations represented by the alphabet $\Sigma=\{\alpha, \beta\}$, such that the associated words belong to a language $\mathcal{L}_{E}$.  When $\nu=2$ the language is infinite. The values of $\nu$  and  $d$ are
      \bigskip\par\noindent    
   (a) $t=1$ ($n\in {\Bbb{Z}}^{\geq 0},  m\in {\Bbb{Z}}^{\geq 1}$): 
   
   (a1) $\nu=2$, $d=9+3h$, $h\in {\Bbb{Z}}^{\geq 0}$, for $n=0, m=1$.
   
   (a2) $\nu=3(n+m)-1$, $d=3(n+3m(n+m))+h(\nu+1)$, $h=0, \dots 3(n+m)-2$, for $n\in {\Bbb{Z}}^{\geq 1},  m\in {\Bbb{Z}}^{\geq 1}$. 

\par\noindent  
   (b) $t=3$ ($n\in {\Bbb{Z}}^{\geq 0}$): 
   
   (b1) $\nu=3(n+m)-1$ , $d=3(3m(n+m)-n-2m+(n+m)h+l+1)$, $h=0, \dots 3(n-l+m-1)$, $m\in {\Bbb{Z}}^{\geq 2}$, $l=0, \dots m-2$.
   
     (b2) $\nu=3(n+m)+2$ , $d=3(n(3m+1)+3m(m+1)+h(n+m+1)+l+1)$, $h=0, \dots 3(n-l+m)+1$, $m\in {\Bbb{Z}}^{\geq 1}$, $l=0$ if $m=1$ and $l=0, \dots m-2$ if $m\in {\Bbb{Z}}^{\geq 2}$.
   
       (b3) $\nu=3(n+m)+2$ , $d=3(3nm+m(3m+2)+h(n+m+1)+l+1)$, $h=0, \dots 3(n-l+m)-1$, $m\in {\Bbb{Z}}^{\geq 2}$, $l=0, \dots m-2$.
     
     \end{lem}

 \begin{proof}
      \bigskip\par
     We apply alternatively the two transformations $\alpha$, $\beta$ and get trees with corresponding words $\alpha\beta\alpha \dots$. When we apply the transformations $h=0,1,2, \dots$ times, namely  $|\alpha\beta\alpha \dots|=h$, to a tree linked to ${\mathcal{B}}^{(t)}_{d_{0},\nu,\epsilon}(w)$ the resulting polynomials have degrees  $d=d_{0}+h(\nu+1)$.  We are interested in polynomials having degrees multiple by 3. This occurs if $\nu+1=3r$, which is possible in the following cases: (1) ${\mathcal{B}}^{(1)}_{3(n+3m(n+m)), 3(n+m)-1}(w)$; (2) ${\mathcal{B}}^{(3)}[x=1, p=3m-1, n, m, l]$, ${\mathcal{B}}^{(3)}[x=2, p=3m+1, n, m, l]$ and  ${\mathcal{B}}^{(3)}[x=3, p=3m, n, m, l]$.
\par  
 Admissible words for $\nu=2$ are $\alpha, \alpha\beta, \alpha\beta\alpha,  (\alpha\beta)^2, \dots $ and the language $\mathcal{L}_{E}$  is infinite.
For $\nu>2$ admissible words  can also be obtained by alternating $\alpha$ and $\beta$: $\alpha, \alpha\beta, \alpha\beta\alpha,  (\alpha\beta)^2, \dots  , w$, where 
 $w=(\alpha\beta)^{y/2}$ if  $y$ is even and  $w=(\alpha\beta)^{\lfloor y/2 \rfloor }\alpha$ if $y$ is odd. The value of $y$ is, as we see below,  $y=3(n+m)-2$ when the initial tree is ${\mathcal{B}}^{(1)}[n, m]$ and $y= 3(n-l)+p-x-3\lfloor \frac{1}{x}\rfloor+2$, $x=1, 2, 3$, for ${\mathcal{B}}^{(3)}[x, p, n, m, l]$.  
\bigskip\par\noindent
   (a) $t=1$. For  $n\in {\Bbb{Z}}^{\geq 0},  m\in {\Bbb{Z}}^{+}, d_{0}=3(n+3m(n+m))$ the polynomial ${\mathcal{B}}^{(1)}_{d_{0},\nu}(w)$, has $3m$ critical points of multiplicity $\nu=3(n+m)-1$ with critical value $\zeta=-1$ and one critical point of multiplicity $\nu$ with critical value $\zeta=1$. After applying the transformations $h=0,1,2, \dots y$ times, the resulting ${\mathcal{G}}^{(1)}_{d,\nu}(w)$  has $3m+h$ critical points of multiplicity $\nu$ with critical value $\zeta=-1$ and one critical point of multiplicity $\nu$ with critical value $\zeta=1$. It also has  $\lfloor \frac{h}{2} \rfloor$  critical points with critical value $\zeta=1$ of multiplicity $\nu=2$ and, when $h$ is odd, one additional critical point with critical value $\zeta=1$ of multiplicity $\nu=1$. Hence  the polynomials ${\mathcal{G}}^{(1)}_{d,\nu}(w)$, $\nu>2$ have 
\bigskip\par
$d=d_{0}+h(\nu+1)=3(n+3m(n+m))+h(\nu+1)$, $\nu=3(n+m)-1$, 
\bigskip\par\noindent
and $N_{-1}=3m+h, N_{1}=1$, therefore  $\lfloor \frac{d}{\nu+1}\rfloor=N_{-1}$ and  $\lfloor \frac{d-1}{\nu}\rfloor=N_{-1}+1$ only if $h \le 3(n+m)-2$. 
\par
The case ${\mathcal{G}}^{(1)}_{d,\nu}(w)$ with $\nu=2$ occurs when $n=0$ and $m=1$. It has $d=9+3h, N_{-1}(\mathcal{G}, 2)=3+h$ and $N_{1}(\mathcal{G}, 2)=1+\lfloor\frac{h}{2}\rfloor$, but $\lfloor \frac{d}{\nu+1}\rfloor=N_{-1}$ and $\lfloor \frac{d-1}{\nu}\rfloor-\lfloor \frac{d}{\nu+1}\rfloor=1+\lfloor\frac{3h}{2}\rfloor-h=N_{1}$, which are valid for any $h\in {\Bbb{Z}}^{\geq 0}$, hence the language is infinite.

\bigskip\par\noindent
   (b) $t=3$. Here $N_{1}=1$ in all the cases. When ${\mathcal{G}}^{(3)}_{d_{0},\nu,\epsilon}(w)={\mathcal{B}}^{(3)}[x=1, p=3m-1, n, m, l]$, $\nu=3(n+m)-1$, the successive application of the transformations  gives polynomials with $d=d_{0}+h(\nu+1)=3(3m(n+m)-n-2m+(n+m)h+l+1)$, $N_{-1}=p-1+h=3m-2+h$. Condition $(E)$ is satisfied  if $h \le 3(n+m-l-1)=3(n-l)+p-2$. 
\par
If ${\mathcal{G}}^{(3)}_{d_{0},\nu,\epsilon}(w)={\mathcal{B}}^{(3)}[x=2, p=3m+1, n, m, l]$, $\nu=3(n+m)+2$, then the polynomials have 
$d=d_{0}+h(\nu+1)=3(n(3m+1)+3m(m+1)+h(n+m+1)+l+1)$, $N_{-1}=p-1+h=3m+h$, and condition $(E)$ is fulfilled when $h \le 3(n+m-l)+1=3(n-l)+p$. 
\par
For ${\mathcal{G}}^{(3)}_{d_{0},\nu,\epsilon}(w)={\mathcal{B}}^{(3)}[x=3, p=3m, n, m, l]$, $\nu=3(n+m)+2$, the polynomials have  $d=d_{0}+h(\nu+1)=3(3nm+m(3m+2)+h(n+m+1)+l+1)$, $N_{-1}=p-1+h=3m-1+h$,
hence condition $(E)$  is satisfied  only if $h \le 3(n+m-l)-1=3(n-l)+p-1$. 

      \end{proof}
                 \begin{figure}[h]
 \includegraphics[width=20pc]{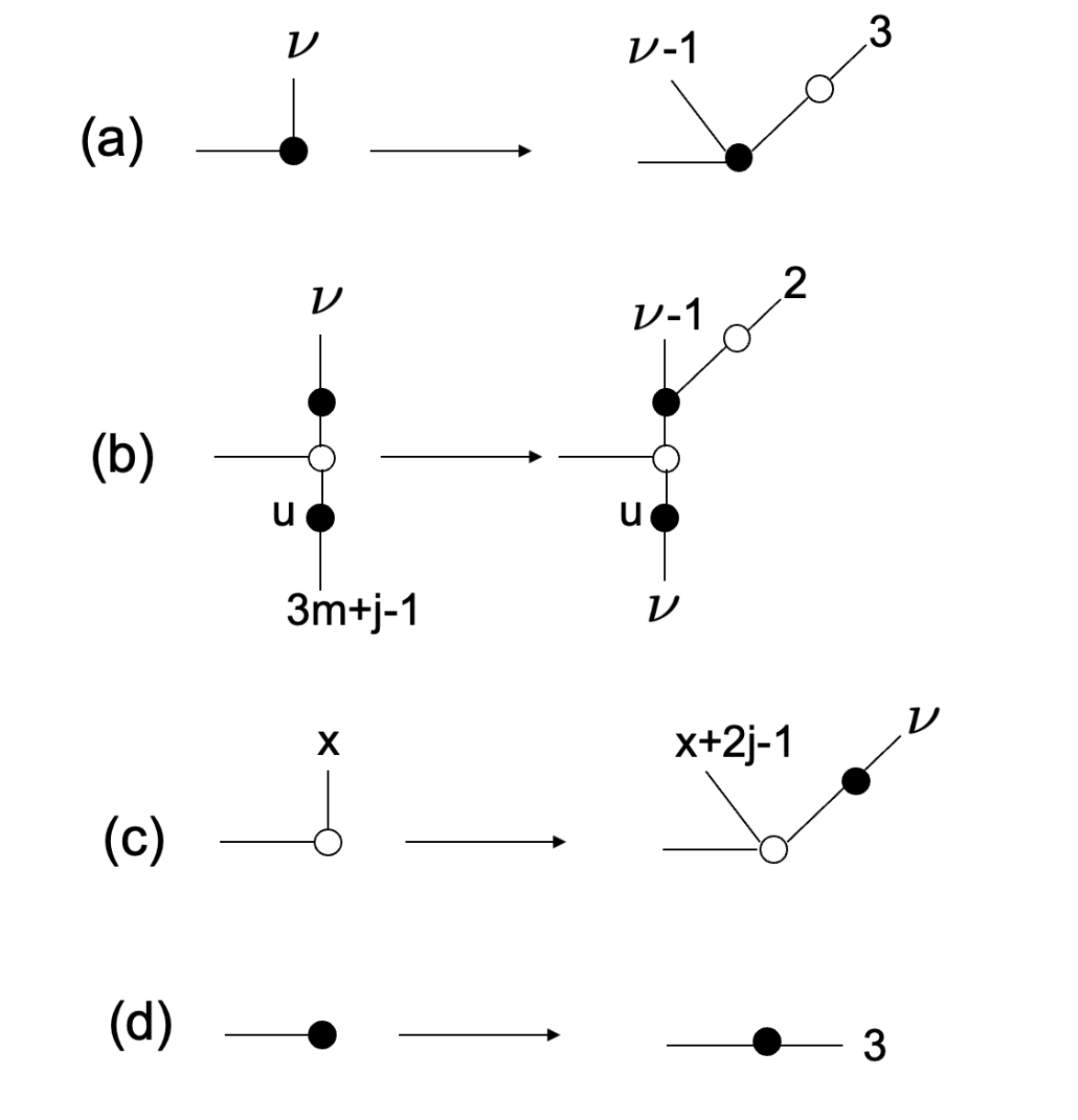}
\caption{\label{label} The 5 types of transformations for ${\mathcal{G}}^{(2)}$: (a) $\alpha$, (b) $\beta$ , (c) $\gamma$, (d) $\delta$, (e) $\bar{\delta}$ is as $\delta$ but applied to a white vertex.}
\end{figure}  
 \bigskip\par

          \begin{figure}[h]
 \includegraphics[width=30pc]{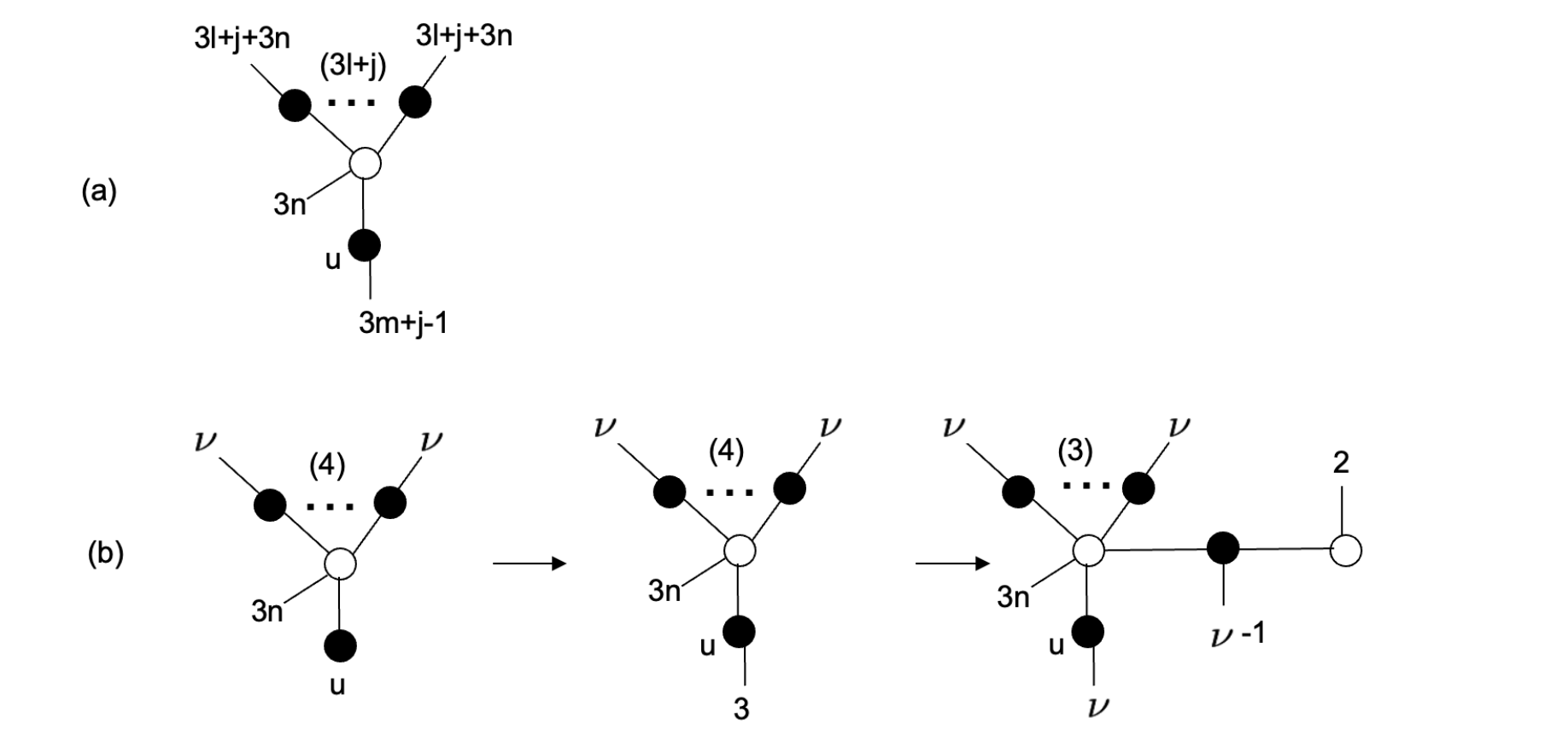}
\caption{\label{label} (a) Tree for  ${\mathcal{B}}^{(2)}_{d,  \nu,\epsilon}(w)$, where the number of edges with label $3l+j+3n$ is indicated in brackets above the suspension points. (b) Applying  the transformations $\delta, \beta$ to the tree with $j=1,m=0, l=1$ corresponding to ${\mathcal{B}}^{(2)}_{d,  \nu,\epsilon}(w), d=15n+21, \nu= 3n+4,\epsilon= 0$ to obtain ${\mathcal{G}}^{(2)}$.}
\end{figure}

          \begin{figure}[h]
 \includegraphics[width=33pc]{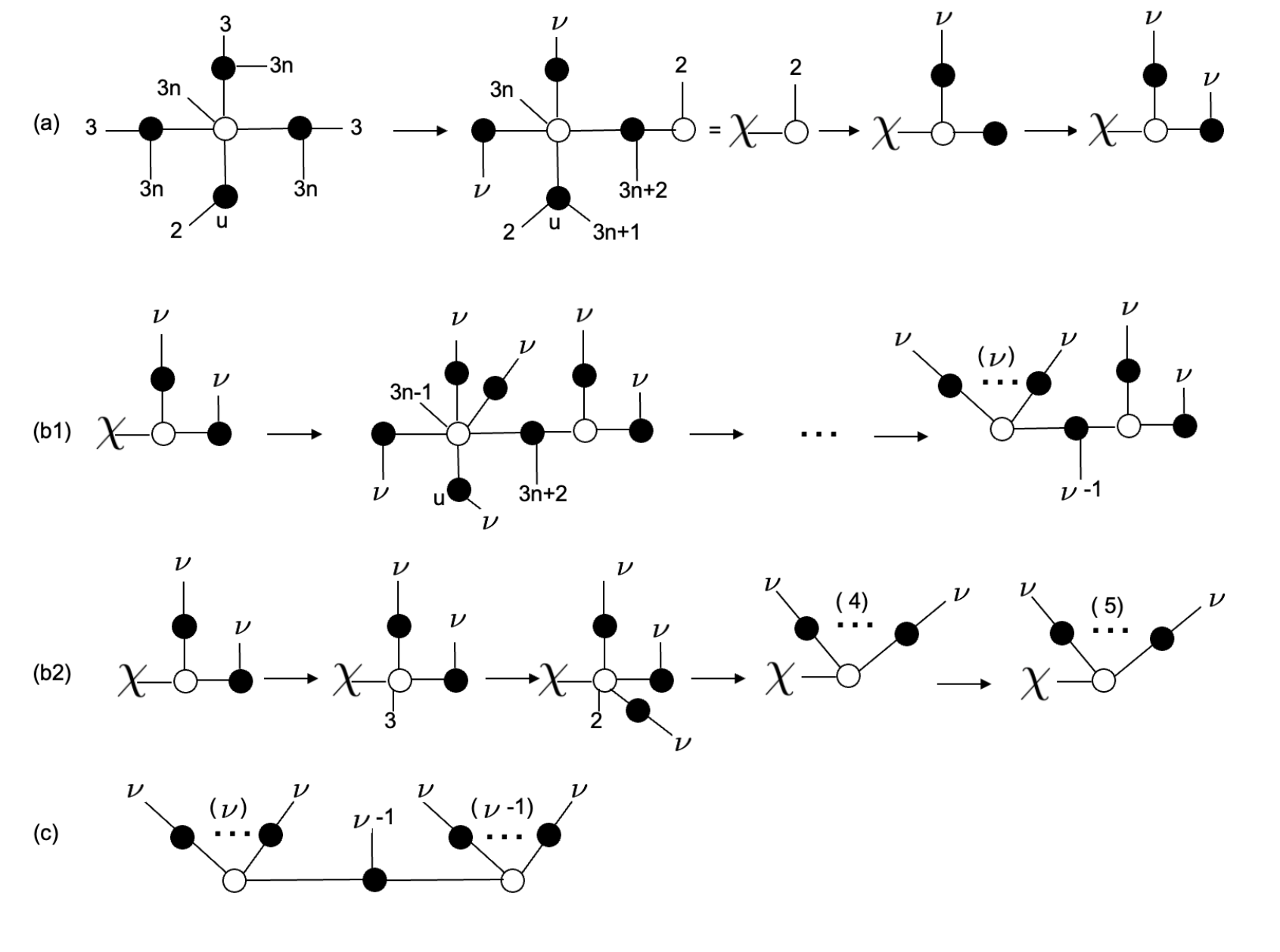}
\caption{\label{label} Applying  the transformations to the tree corresponding to ${\mathcal{B}}^{(2)}_{d,  \nu,\epsilon}(w), j=0, m=1, d=12n+5, \nu= 3n+3,\epsilon= 2 $, to obtain ${\mathcal{G}}^{(2)}$. (a) The first three steps. The process can be continued in two ways: (b1) and (b2). The last tree is represented in (c).}
\end{figure}   
\par 
For the derivation of ${\mathcal{G}}^{(2)}_{d,\nu, \epsilon}(w)$, we characterise a tree corresponding to $A_{\nu}$ by $\mathcal{T}_{\nu}(d, N_{-1})=\mathcal{T}_{\nu}(d, N_{-1}, 1)$, because $N_{1}=1$ in all cases. The initial tree is $\mathcal{T}_{\nu,0}=\mathcal{T}_{\nu}(d_{0}, 3l+j), d_{0}=(3l+j+1)\nu+3m+j, \nu=3n+3l+j$ (see Lemma 2 in \cite{esc26} and Fig. 4(a)).
\par
We consider 5 types of transformations (Fig.3):
\par\noindent   
   \hspace{2cm} $\alpha: \mathcal{T}_{\nu}(d, N_{-1}) \longrightarrow \mathcal{T}_{\nu}(d+3, N_{-1})$, 
\par\noindent
   \hspace{2cm} $\beta: \mathcal{T}_{\nu}(d, N_{-1}) \longrightarrow \mathcal{T}_{\nu}(d+3(n+l-m+1), N_{-1}+1)$, 
\par\noindent
   \hspace{2cm} $\gamma: \mathcal{T}_{\nu}(d, N_{-1}) \longrightarrow \mathcal{T}_{\nu}(d+\nu, N_{-1}+1)$, 
\par\noindent
   \hspace{2cm} $\delta: \mathcal{T}_{\nu}(d, N_{-1}) \longrightarrow \mathcal{T}_{\nu}(d+3, N_{-1})$,
\par\noindent
   \hspace{2cm} $\bar{\delta}: \mathcal{T}_{\nu}(d, N_{-1}) \longrightarrow \mathcal{T}_{\nu}(d+3, N_{-1})$.

      \bigskip\par    
      
           \begin{lem}  There exist polynomials  ${\mathcal{G}}^{(2)}_{d,\nu,\epsilon}(w)$ whose trees are obtained from initial trees of type ${\mathcal{B}}_{d,\nu,\epsilon}^{(2)}$ by means of a series of transformations such that the associated words belong to a language $\mathcal{L}_{E}$ with alphabet $\Sigma=\{\alpha, \beta, \gamma, \delta, \bar{\delta} \}$. The possible series include
        \bigskip\par\noindent             
           (a) $j=0$ 
           
           (a1) $l=m=1, n\in {\Bbb{Z}}^{\geq 1}$, $\nu=3n+3$, $d=\nu(3r+k+4)+3(r+1)$, $r=1, \dots n$, $k=0, \dots \nu$.
           
           (a2) $l=m+s, s \in {\Bbb{Z}}^{\geq 0}, n, m \in {\Bbb{Z}}^{\geq 1}$, $\nu=3(n+m)$,   $d_{0}=3m(\nu+1)+\nu$, $d \in \{d_{0}+3k(n+1)\}_{k=0,1}\cup$ 
           
           $\{d_{0}+3(n+1)+3q\}_{q=1,\dots m+s-1}\cup\{d_{0}+3(n+1)+3(m+s-1)+r\nu\}_{r=1, \dots \nu-1}\cup$
           
           $\{d_{0}+3(n+1)+3(m+s-1)+\nu(\nu-1)+3f\}_{f=1, \dots n}\cup$ 
           
           $\{d_{0}+3(n+1)+3(m+s-1)+\nu(\nu-1)+3n+g\nu\}_{g=1, \dots 3n}$.

          \par\noindent    
          (b) $j=1$ 
          
          (b1) $m=0, l=1$, $\nu=3n+4$, $d=3(n(5+\lfloor\frac{k}{2}\rfloor)+k+7)$, $k=0,1, 2$.
          
           (b2) $l=m, n \in {\Bbb{Z}}^{\geq 1}$, $d_{0}=3((n+m)(3m+2)+2m+1)$,
           
           $d \in \{d_{0}+3k(n+1)\}_{k=0,1}\cup\{d_{0}+3(n+1)+9q\}_{q=1,\dots m-1}\cup\{d_{0}+3(n+1)+9(m-1)+r\nu\}_{r=1,2}$.
           
            (b3) $l=m+s, s \in {\Bbb{Z}}^{\geq 1}, n \in {\Bbb{Z}}^{\geq 0}$, $d_{0}=3((n+m+s)(3(m+s)+2)+2m+s+1)$, 
           
           $d \in \{d_{0}+3k(n+1)\}_{k=0,1}\cup\{d_{0}+3(n+1)+9q\}_{q=1,\dots m+s-1}\cup\{d_{0}+3(n+1)+9(m+s-1)+r\nu\}_{r=1,2}$.
          
     \end{lem}
     
      \begin{proof}
 \bigskip\par 
(a) $j=0$. Before treating the general cases we study the series with the initial tree of ${\mathcal{B}}^{(2)}_{d_{0}, \nu, \epsilon}(w), l=1, m=1, d_{0}=12n+15, \nu=3n+3, \epsilon=2$ (Fig. 3(a), right in \cite{esc26}). The first steps are represented in Fig. 5. We first add $\nu-2$ white vertices to the vertex $u$ and 2 additional  to one of the leaf white vertex as in Fig. 5(a) which is the transformation $\beta$. Then we apply 2 times the transformation $\gamma$: the words are $\beta\gamma^q$, $q=1,2$ with corresponding trees $\mathcal{T}_{\nu}(d_{0}+(q+1)\nu, 4+q)$. 
\par
We can apply two types of transformations to the tree corresponding to $\beta\gamma^2$. In Fig. 5(b1) we apply $\nu-3$ times the substitution rule $\gamma$  and we obtain $\beta\gamma^2\gamma^r$, $r=1,2 \dots \nu-3$. The trees are $\mathcal{T}_{\nu}(d_{0}+k \nu, 3+k)$, $k=4, 5, \dots , \nu$. The case $k=4$ is represented in Fig. 5(b1). The last tree of the whole series, represented in Fig. 5(c), has $d=2\nu(\nu+1)$ therefore we have to check if $d_{0}+ \nu^2<2\nu(\nu+1)$ in order to follow the steps represented in Fig. 5(b2). We then add 3 black vertices to the white vertex corresponding to the critical value with lower multiplicity as in Fig. 5(b2). The word is  $\beta\gamma^2\bar{\delta}$ with tree $\mathcal{T}_{\nu}(d_{0}+3 \nu+3, 6)$ and, after applying three times $\gamma$, we continue as above in Fig. 5(b1). The sequence of allowed words we get is 
       \bigskip\par    
$\beta$,  $\beta\gamma^{q}, (q=1,2, \dots \nu-1)$, $\beta\gamma^{2}\bar{\delta}$, $\beta\gamma^{2}\bar{\delta}(\gamma^{3}\bar{\delta})^{r}\gamma^{s}, r=0, 1, \dots  n-1, s=0,1, \dots \nu$
       \bigskip\par\noindent  
and the series of trees is $\mathcal{T}_{\nu}(d_{0}+(3r+k) \nu+3r, 3+3r+k), r= 1, 2, \dots , n, k=0, 1, \dots,  \nu$. 
\par
For the cases ${\mathcal{B}}^{(2)}[j, n, m, l]$ with $j=0, l=m$ the initial tree corresponds to ${\mathcal{G}}^{(2)}_{d_{0},\nu,\epsilon}(w)={\mathcal{B}}^{(2)}_{d_{0},\nu,\epsilon}(w)$ with $d_{0}=3m(\nu+1) +\nu, \nu=3n+3m, \epsilon=3m-1$. Allowed words are
       \bigskip\par     
        $\{\beta\}$ $\cup$ $\{\beta\alpha^q\}_{q=1, 2, \dots m-1}$ $\cup$ $\{\beta\alpha^{m-1}\gamma^r\}_{r=1, 2, \dots , 3n, 3n+1, \dots  3n+3m-1}$ $\cup$  $\{\beta\alpha^{m-1}\gamma^{\nu-1}\alpha^f\}_{f=1, 2, \dots  n}$  $\cup$  $\{\beta\alpha^{m-1}\gamma^{\nu-1}\alpha^n\gamma^g\}_{g=1, 2, \dots  3n}$ $\subset \mathcal{L}_{E}$
\par\noindent      
 where  q=0 if $m=1$. 
\par
For $l=m+s, s \in {\Bbb{Z}}^{\geq 1}$  permitted words  are
       \bigskip\par     
        $\{\beta\}$ $\cup$ $\{\beta\alpha^q\}_{q=1, 2, \dots m+s-1}$ $\cup$ $\{\beta\alpha^{m+s-1}\gamma^r\}_{r=1, 2, \dots , 3n, 3n+1, \dots  \nu-1}$ $\cup$  $\{\beta\alpha^{m+s-1}\gamma^{\nu-1}\alpha^f\}_{f=1, 2, \dots  n}$  $\cup$  $\{\beta\alpha^{m+s-1}\gamma^{\nu-1}\alpha^n\gamma^g\}_{g=1, 2, \dots  3n}$ $\subset \mathcal{L}_{E}$. 
\par
The associated degrees are   
$$d \in \{d_{0}+3k(n+1)\}_{k=0,1}\cup\{d_{0}+3(n+1)+3q\}_{q=1,\dots m+s-1}\cup\{d_{0}+3(n+1)+3(m+s-1)+r\nu\}_{r=1, \dots \nu-1}$$
$$\cup\{d_{0}+3(n+1)+3(m+s-1)+\nu(\nu-1)+3f\}_{f=1, \dots n}\cup\{d_{0}+3(n+1)+3(m+s-1)+\nu(\nu-1)+3n+g\nu\}_{g=1, \dots 3n}$$.
\par
Other sequences of transformations are possible. For instance, if $l=m=1, n=2$, we have  
       \bigskip\par          
  $\{\beta\}$ $\cup$  $\{\beta\gamma\}$ $\cup$  $\{\beta\gamma^2\}$ $\cup$ $\{A\gamma^p\}_{p=1, 2, ... 6}$ $\cup$   $\{A\alpha\gamma^q\}_{q=0, 1, 2, 3}$ $\cup$  $\{B\alpha\gamma^r\}_{r=0, 1, 2, 3}$ $\cup$  $\{C\gamma\}$ $\cup$  $\{C\gamma^2\}$ $\cup$  $\{C\alpha\gamma^s \}_{s=0, 1, 2, \dots 9}$ $\subset \mathcal{L}_{E}$
\par\noindent    
  where $A=\beta\gamma^2, B=A\gamma^6, C=A\alpha\gamma^3$. As mentioned above different words may be linked with the same polynomial, like $B\alpha$ and $C\gamma^3$ for $d=123$ or $B\alpha^2, C\gamma^3\alpha, C\alpha\gamma^3$ for $d=126$.          
     \bigskip\par
(b) $j=1$. The case ${\mathcal{G}}^{(2)}_{d_{0}, \nu}(w)={\mathcal{B}}^{(2)}_{d_{0}, \nu, 0}(w)$, $m=0, l=1$, $d_{0}=15n+21, \nu=3n+4,  N_{-1}(\mathcal{G}, \nu)=4$ (Fig. 3(b), right in \cite{esc26}) is represented in Fig. 4(b) and a possible sequence of transformations is
         \begin{equation} 
{\mathcal{G}}^{(2)}_{15n+21, 3n+4}(w){ \delta \atop \longrightarrow } {\mathcal{G}}^{(2)}_{15n+24,3n+4}(w)   { \beta \atop \longrightarrow } {\mathcal{G}}^{(2)}_{18n+27,3n+4}(w)
    \end{equation} 
\par\noindent
with associated words $\delta$ and $\delta\beta$. 
\par
When the initial tree corresponds to $d_{0}= 3(m+l+(n+l)(3l+2)+1)$ with $l=m$, then allowed words are  
        \bigskip\par
$ \beta, \beta\alpha^{3q}$ ($q=1,2, \dots ,m-1$),  $ \beta\alpha^{3m-3}\gamma, \beta\alpha^{3m-3}\gamma^2$ ($n \geq 1$). 
        \bigskip\par\noindent
And when $l=m+s, s \in {\Bbb{Z}}^{\geq 1}$ we have  
     \bigskip\par
$ \beta, \beta\alpha^{3q}$ ($q=1,2, \dots ,m+s-1$),  $ \beta\alpha^{3m+3s-3}\gamma, \beta\alpha^{3m+3s-3}\gamma^2$ ($n \geq 0$)
\par\noindent
with degrees  
$$d \in \{d_{0}+3k(n+1)\}_{k=0,1}\cup\{d_{0}+3(n+1)+9q\}_{q=1,\dots m+s-1}\cup\{d_{0}+3(n+1)+9(m+s-1)+r\nu\}_{r=1,2}$$.
 \end{proof} 
 
           \begin{figure}[h]
 \includegraphics[width=13pc]{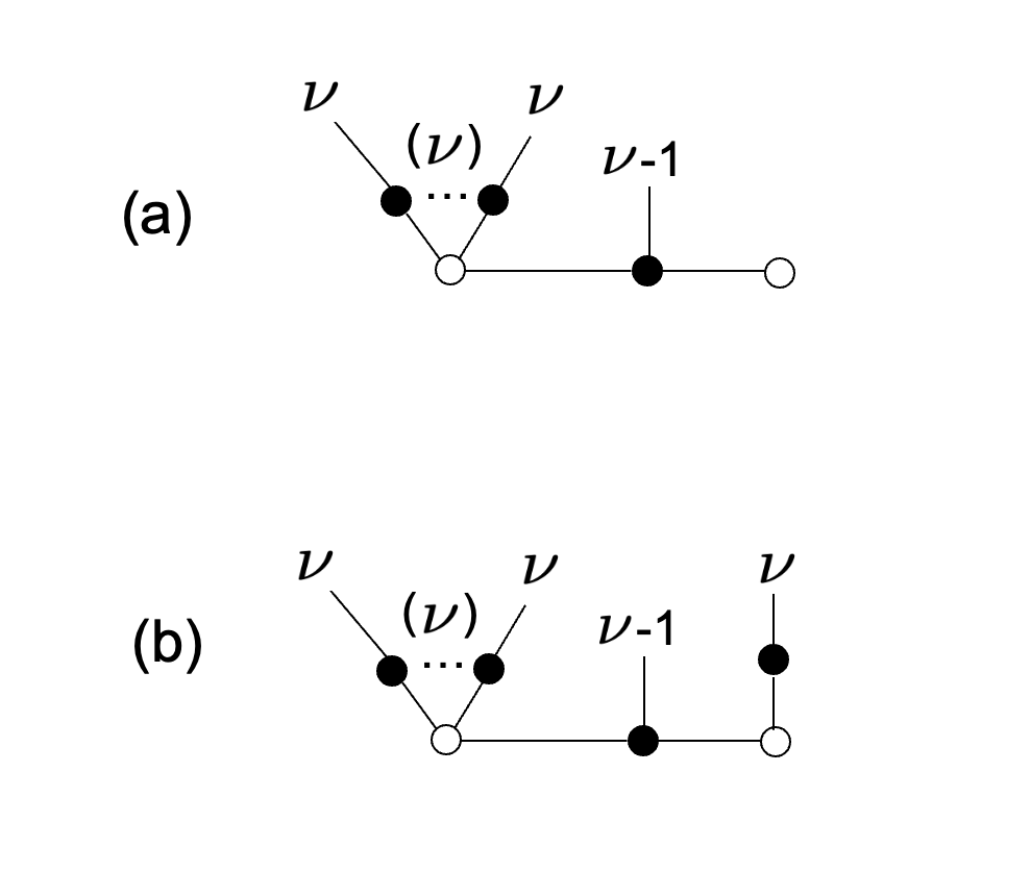}
\caption{\label{label}  Initial trees (a) ${\mathcal{G}}^{(4)}_{d_{0}, \nu}(w), \nu=3n+2$, (b)  ${\mathcal{G}}^{(6)}_{d_{0}, \nu}(w), \nu=3n+4$.}
\end{figure}

We now introduce three families ${\mathcal{G}}^{(t)}_{d, \nu}(w), t=4,5,6$ whose initial trees  ${\mathcal{G}}^{(t)}_{d_{0}, \nu}(w)$ are represented in Fig.6 for $t=4,6$, and Fig.5(a) for $t=5$. 

      \bigskip\par 

           \begin{lem}  The polynomials  ${\mathcal{G}}^{(t)}_{d,\nu}(w), \nu=3n+t-2, n \in {\Bbb{Z}}^{\geq 0},  t= 4, 5, 6$, have the following degrees 
           
           (a) $\nu=3n+2$:  $d(\nu,l,r)=(\nu +1)(l\nu+r+1)$, $l \in {\Bbb{Z}}^{\geq 1}$, $r=0, 1, \dots, 3n+1$.
           
               (b)  $\nu=3n+3$: $d(\nu,l,r)=\nu(\nu+1)+3$ for  $r=-1$, $l=1$; $d(\nu,l,r)=\nu(l(\nu+1)+1)+3$, for $r=0$, $l \in  {\Bbb{Z}}^{\geq 1}$;  $d(\nu,l,r)=\nu(l(\nu+1)+r+1)$ for $r=1, \dots 3n+3 $, $l \in  {\Bbb{Z}}^{\geq 1}$.
            
                     (c) $\nu=3n+4$: $d(\nu,l,r)=(\nu +1)(l (\nu-1)+3(r+\lfloor\frac{l+1}{3}\rfloor +1))$ with $r=-1, 0,1, \dots, n$ if $l =3k+2$ and $r=0,1, \dots, n$ if $l = 3k+a, a\in \{ 0, 1\}$, $l \in  {\Bbb{Z}}^{\geq 1}$, $k \in  {\Bbb{Z}}^{\geq 0}$.
           
  Their trees are obtained from the initial trees ${\mathcal{G}}^{(t)}_{d_{0}, \nu}(w)$ by means of a series of substitution rules such that the associated words belong to an infinite language $\mathcal{L}_{E}$.
                \end{lem}
     
      \begin{proof}
 \bigskip\par 
(a)  The polynomials ${\mathcal{G}}^{(4)}_{d, \nu}(w)$ are defined for $\nu=3n+2, n \in {\Bbb{Z}}^{\geq 0}$, with initial tree having $d_{0}=(\nu+1)^2$ and given in Fig.6(a). We consider the transformations (Fig.7)
      \bigskip\par\noindent
   \hspace{2cm} $\alpha: \mathcal{T}_{\nu}(d, N_{-1},  N_{1}) \longrightarrow \mathcal{T}_{\nu}(d+\nu+1, N_{-1}+1,  N_{1}) $, 
\par\noindent
   \hspace{2cm} $\beta: \mathcal{T}_{\nu}(d, N_{-1},  N_{1})  \longrightarrow \mathcal{T}_{\nu}(d+\nu+1, N_{-1}+1,  N_{1}+1) $, 
\par\noindent
   \hspace{2cm} $\gamma: \mathcal{T}_{\nu}(d, N_{-1},  N_{1}) \longrightarrow \mathcal{T}_{\nu}(d+\nu+1, N_{-1}+1,  N_{1}) $. 

            \begin{figure}[h]
 \includegraphics[width=22pc]{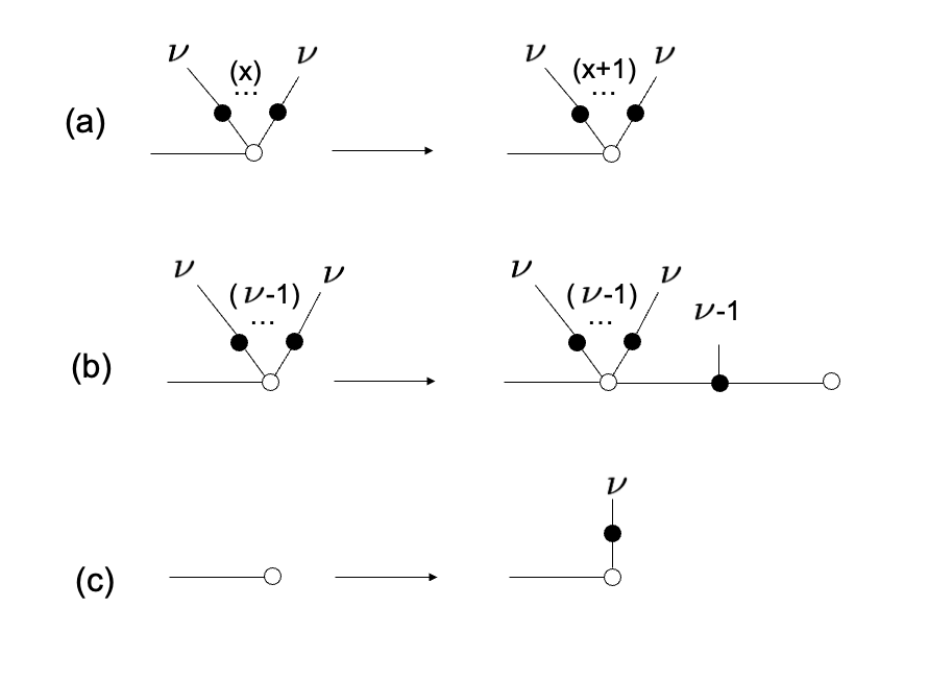}
\caption{\label{label}  Substitution rules for ${\mathcal{G}}^{(4)}_{d, \nu}(w)$, (a) $\alpha$, $1 \leq x < \nu-1$ (b) $\beta$, (c) $\gamma$.  }
\end{figure}
 \bigskip\par 
 
The words of the language with $\Sigma=\{\alpha, \beta, \gamma \}$ corresponding to polynomials satisfying condition $(E)$ are 
              \bigskip\par  
   $\{\gamma\}$ $\cup$   $\{\gamma\alpha^r\}_{r= 1, 2, \dots 3n}$ $\cup$ $\{A \beta\}$ $\cup$ $\{A\beta\gamma\}$ 
  $\cup$ $\{A\beta\gamma\alpha^r\}_{r= 1, 2, \dots 3n}$ $\cup$ $\{A \beta A \beta \}$ $\cup$ $\dots$  
                \bigskip\par  
  $\cup$ $\{A (\beta A)^{p-1}\beta \}$ 
  $\cup$ $\{A (\beta A)^{p-1}\beta\gamma \}$   $\cup$ $\{A (\beta A)^{p-1}\beta\gamma \alpha^r\}_{r= 1, 2, \dots 3n}$ $\cup$ $\{A (\beta A)^{p} \beta\}$ $\cup$ 
  $\dots$ $\subset \mathcal{L}_{E} $         
\bigskip\par\noindent     
where $A= \gamma\alpha^{3n}$ and $p \in {\Bbb{Z}}^{\geq 1}$.    
We get the trees $\mathcal{T}_{\nu}(d, N_{-1},  N_{1})$  with $d=(\nu +1)(l\nu+r+1)$, $N_{-1}=l\nu+r+1$, $N_{1}=l$, $l \in {\Bbb{Z}}^{\geq 1}$, $r=0, 1, \dots, 3n+1$.

   \bigskip\par 
(b) ${\mathcal{G}}^{(5)}_{d, \nu}(w)$ are defined for $\nu=3n+3, n \in {\Bbb{Z}}^{\geq 0}$ with initial tree as in Fig.5(a) having $d_{0}=(\nu+1)\nu+3$.  The trees $\mathcal{T}_{\nu}(d, N_{-1},  N_{1})$  have  

  \hspace{0.5cm} (b1) $d(\nu,l,r)=\nu(\nu+1)+3$, $N_{-1}=\nu$, $N_{1}=1$, for $l=1$ and $r=-1$.

  \hspace{0.5cm} (b2) $d(\nu,l,r)=\nu(l(\nu+1)+1)+3$, $N_{-1}=l\nu+1$, $N_{1}=l$ for $r=0$, $l \in  {\Bbb{Z}}^{\geq 1}$.

  \hspace{0.5cm} (b3) $d(\nu,l,r)=\nu(l(\nu+1)+r+1)$, $N_{-1}=l\nu+r$, $N_{1}=l$ for $r=1, \dots \nu $, $l \in  {\Bbb{Z}}^{\geq 1}$.
   \bigskip\par 
The transformations for $t=5$ are represented in Fig.8
      \bigskip\par\noindent
   \hspace{2cm} $\alpha: \mathcal{T}_{\nu}(d, N_{-1},  N_{1}) \longrightarrow \mathcal{T}_{\nu}(d+\nu, N_{-1}+1,  N_{1}) $, 
\par\noindent
   \hspace{2cm} $\beta: \mathcal{T}_{\nu}(d, N_{-1},  N_{1})  \longrightarrow \mathcal{T}_{\nu}(d+\nu-3, N_{-1},  N_{1}) $, 
\par\noindent
   \hspace{2cm} $\gamma: \mathcal{T}_{\nu}(d, N_{-1},  N_{1}) \longrightarrow \mathcal{T}_{\nu}(d+\nu, N_{-1}+1,  N_{1}) $, 
\par\noindent
   \hspace{2cm} $\delta: \mathcal{T}_{\nu}(d, N_{-1},  N_{1}) \longrightarrow \mathcal{T}_{\nu}(d+\nu+3, N_{-1}+1,  N_{1}+1) $. 

            \begin{figure}[h]
 \includegraphics[width=17pc]{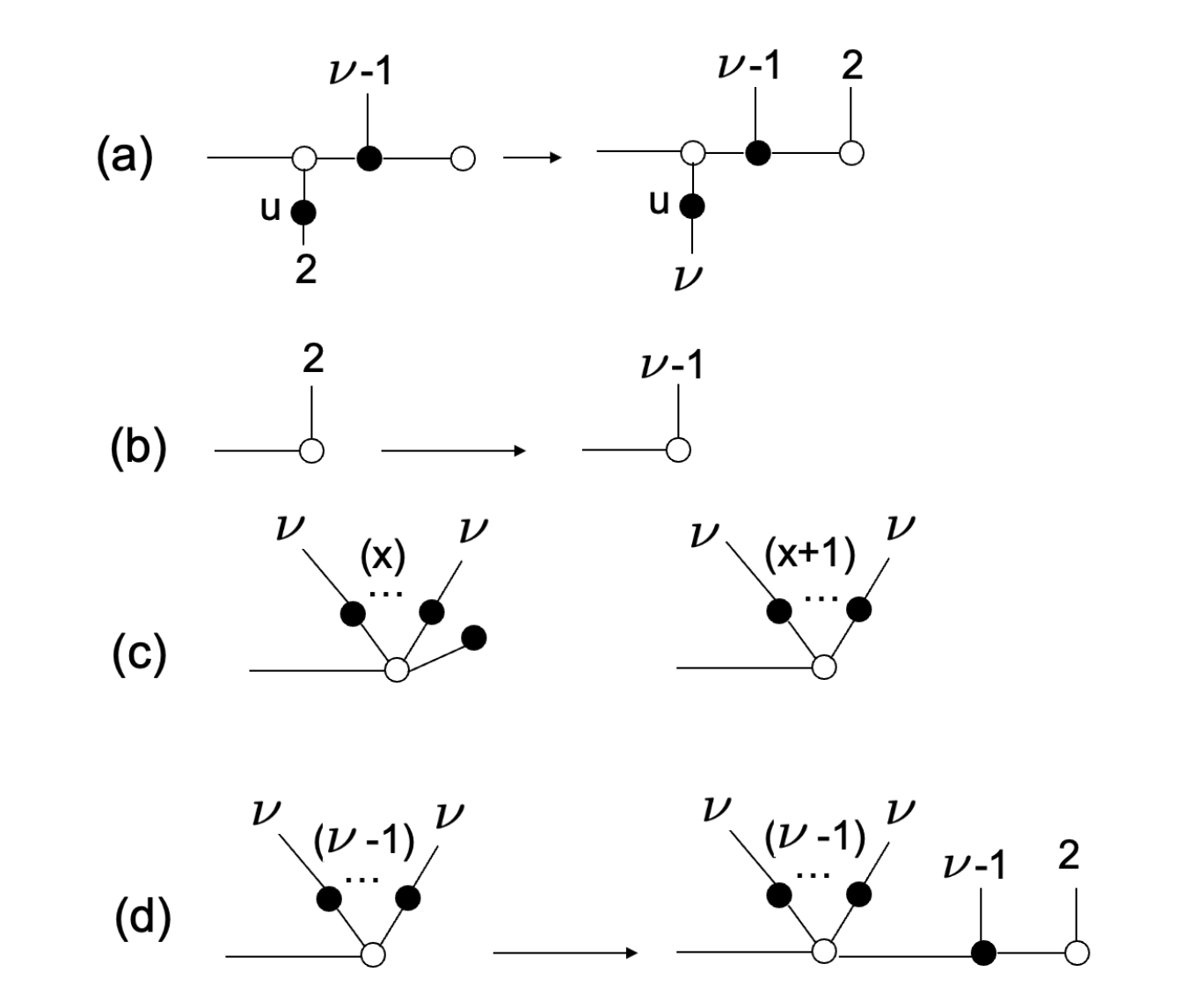}
\caption{\label{label}  Transformations for ${\mathcal{G}}^{(5)}_{d, \nu}(w)$, (a) $\alpha$, (b) $\beta$, (c) $\gamma$, (d) $\delta$.  }
\end{figure}
      \bigskip\par 

 The alphabet is then $\Sigma=\{\alpha, \beta, \gamma, \delta \}$ and the allowed words are
    \bigskip\par 
   $\{\alpha\}$ $\cup$   $\{\alpha\beta\gamma^r\}_{r= 0, 1, \dots, \nu-1}$  $\cup$ $\{\alpha A\}$ $\cup$ $\{\alpha A \beta\}$ $\cup$ \dots  $\cup$ $\{\alpha A^{p}\}$ $\cup$ $\{\alpha A^{p}\beta\}$ $\cup$ \dots  $\subset \mathcal{L}_{E} $ 
     \bigskip\par \noindent 
   where $A=\beta\gamma^{\nu-1}\delta$ and $p \in {\Bbb{Z}}^{\geq 1}$.

   \bigskip\par 
(c) ${\mathcal{G}}^{(6)}_{d, \nu}(w)$ are defined for $\nu=3n+4, n \in {\Bbb{Z}}^{\geq 0}$ with initial tree as in Fig.6(b) having $d_{0}=(\nu+1)(\nu+2)$. The transformations for $t=6$ are (Fig.9)
      \bigskip\par\noindent
   \hspace{2cm} $\alpha: \mathcal{T}_{\nu}(d, N_{-1},  N_{1}) \longrightarrow \mathcal{T}_{\nu}(d+3(\nu+1), N_{-1}+3,  N_{1}) $, 
\par\noindent
   \hspace{2cm} $\beta: \mathcal{T}_{\nu}(d, N_{-1},  N_{1})  \longrightarrow \mathcal{T}_{\nu}(d+3(\nu+1), N_{-1}+3,  N_{1}+1) $, 
\par\noindent
   \hspace{2cm} $\gamma: \mathcal{T}_{\nu}(d, N_{-1},  N_{1}) \longrightarrow \mathcal{T}_{\nu}(d+3(\nu+1), N_{-1}+3,  N_{1}) $, 
\par\noindent
   \hspace{2cm} $\delta: \mathcal{T}_{\nu}(d, N_{-1},  N_{1}) \longrightarrow \mathcal{T}_{\nu}(d+3(\nu+1), N_{-1}+3,  N_{1}+1) $, 
\par\noindent
   \hspace{2cm} $\bar{\delta} : \mathcal{T}_{\nu}(d, N_{-1},  N_{1}) \longrightarrow \mathcal{T}_{\nu}(d+3(\nu+1), N_{-1}+3,  N_{1}+1) $.

            \begin{figure}[h]
 \includegraphics[width=20pc]{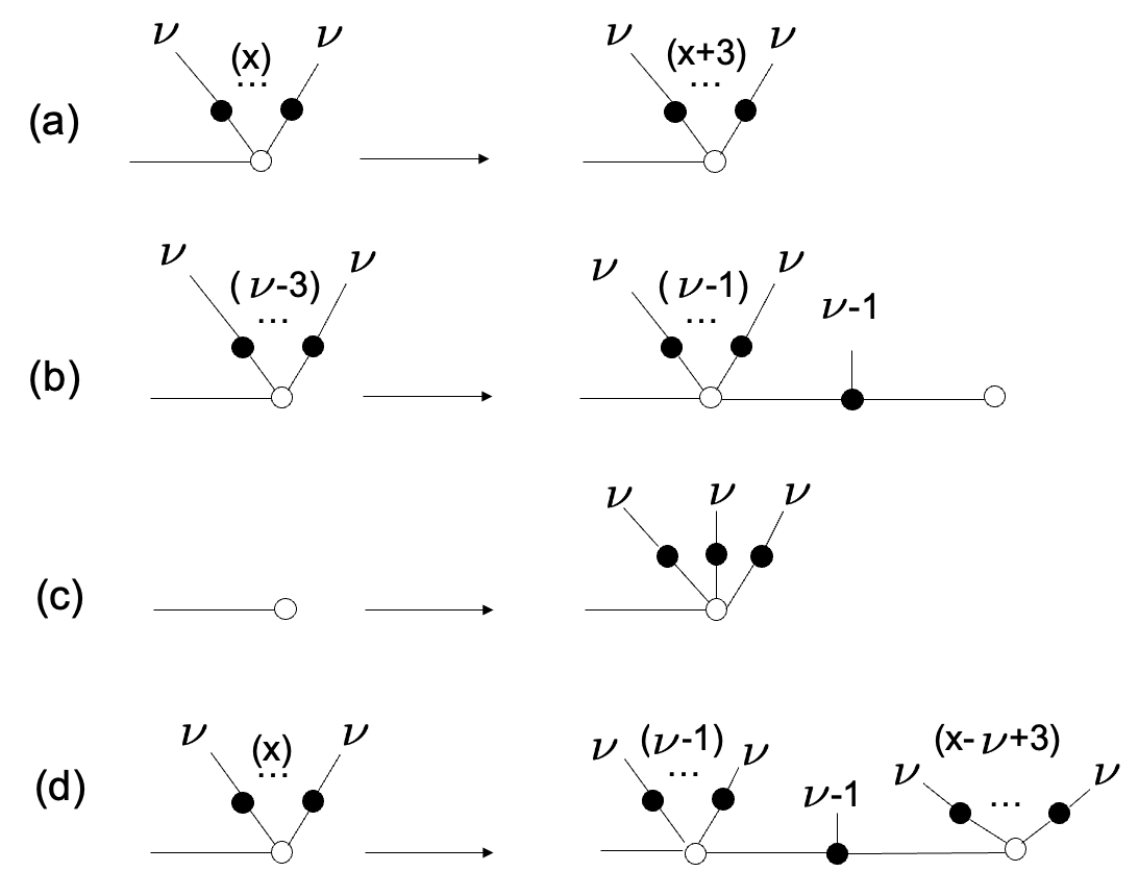}
\caption{\label{label}  Transformations for ${\mathcal{G}}^{(6)}_{d, \nu}(w)$, (a) $\alpha$, $1 \leq x < \nu-3$, (b) $\beta$, (c) $\gamma$, (d) $\delta$ if $x=\nu-1$ and $\bar{\delta}$ if  $x=\nu-2$.   }
\end{figure}

For $n=0$ the allowed words from the alphabet $\Sigma=\{\alpha, \beta, \gamma, \delta, \bar{\delta} \}$ are 
       \bigskip\par 
       $\beta$,  $\beta\gamma$,  $\beta\gamma\delta$, $A$,  $\dots$, $A^{p-1}$, $A^{p-1}\beta$, $A^{p-1}\beta \gamma$, $A^{p-1}\beta\gamma \delta$, $A^{p}$, $\dots$            
 \bigskip\par\noindent
   where $A=\beta\gamma\delta\bar{\delta}$ and $p \in {\Bbb{Z}}^{\geq 2}$. The trees $\mathcal{T}_{4}(d, N_{-1},  N_{1})$  have  $d=15(l+r+\lfloor\frac{l+1}{3}\rfloor +1)$, $N_{-1}=3(l+r+\lfloor\frac{l+1}{3}\rfloor +1)$, $N_{1}=l$, with  $l \in  {\Bbb{Z}}^{\geq 1}$,   $r=-1, 0$ if $l =3k+2$ and $r=0$ if $l = 3k+a$, $k \in  {\Bbb{Z}}^{\geq 0}$, $a\in \{ 0, 1\}$.  The allowed words for $n \geq 1$ are
     \bigskip\par 
  $\{\alpha^r\}_{r= 1, 2, \dots, n}$ $\cup$ $\{ \alpha^{n}\beta\}$ $\cup$ $\{ \alpha^{n}\beta \gamma\alpha^r\}_{r= 0, 1, 2, \dots, n}$ $\cup$ 
  $\{ \alpha^{n}\beta \gamma\alpha^{n}\delta\alpha^r\}_{r= 0, 1, 2, \dots, n}$ $\cup$ $\{ A\}$  $\cup$  $\dots$  
   \bigskip\par
  $\cup$ $\{ A^{p-1}\alpha^r\}_{r= 0, 1, 2, \dots, n} $ $\cup$ $\{ A^{p-1}\alpha^n\beta\} $ $\cup$ $\{ A^{p-1}\alpha^{n}\beta\gamma\alpha^r\}_{r= 0, 1, 2, \dots, n}$ 
  $\cup$ $\{ A^{p-1}\alpha^{n}\beta \gamma\alpha^{n}\}  $  
   \bigskip\par
  $\cup$ $\{ A^{p-1}\alpha^{n}\beta \gamma\alpha^{n}\delta\alpha^r\}_{r= 0, 1, 2, \dots, n}$ $\cup$ $\{ A^{p}\alpha^r\}_{r= 0, 1, 2, \dots, n} $ $\cup$ 
  $\dots$ $\subset \mathcal{L}_{E} $ 
          \bigskip\par\noindent
 where $A=\alpha^{n}\beta \gamma\alpha^{n}\delta\alpha^{n}\bar{\delta}$. Having in mind that $l \in  {\Bbb{Z}}^{\geq 1}$ and $k \in  {\Bbb{Z}}^{\geq 0}$,
the trees $\mathcal{T}_{\nu}(d, N_{-1},  N_{1})$  have   $d=(\nu +1)(l (\nu-1)+3(r+\lfloor\frac{l+1}{3}\rfloor +1))$, $N_{-1}=l (\nu-1)+3(r+\lfloor\frac{l+1}{3}\rfloor +1)$, $N_{1}=l$, with $r=-1, 0,1, \dots, n$ if $l =3k+2$ and $r=0,1, \dots, n$ if $l = 3k+a, a\in \{ 0, 1\}$. 

 \end{proof}

      \bigskip\par

 In order to include the degrees divisible by 3 we make an interpolation between two consecutive degrees $d_{1}$ and $d_{2}$, whose trees are known from the Lemmas, and analyse trees associated to degrees $d_{1}+3i$, where $i=0, \dots, m$, $d_{2}=d_{1}+3m$. We study the cases ${\mathcal{G}}^{(t)}_{d,\nu}(w), t=4,5,6$ where the construction is asymptotic ($d$ is not bounded by $\nu$) and we add additional polynomials with initial trees given in Fig.10.
 
             \begin{figure}[h]
 \includegraphics[width=18pc]{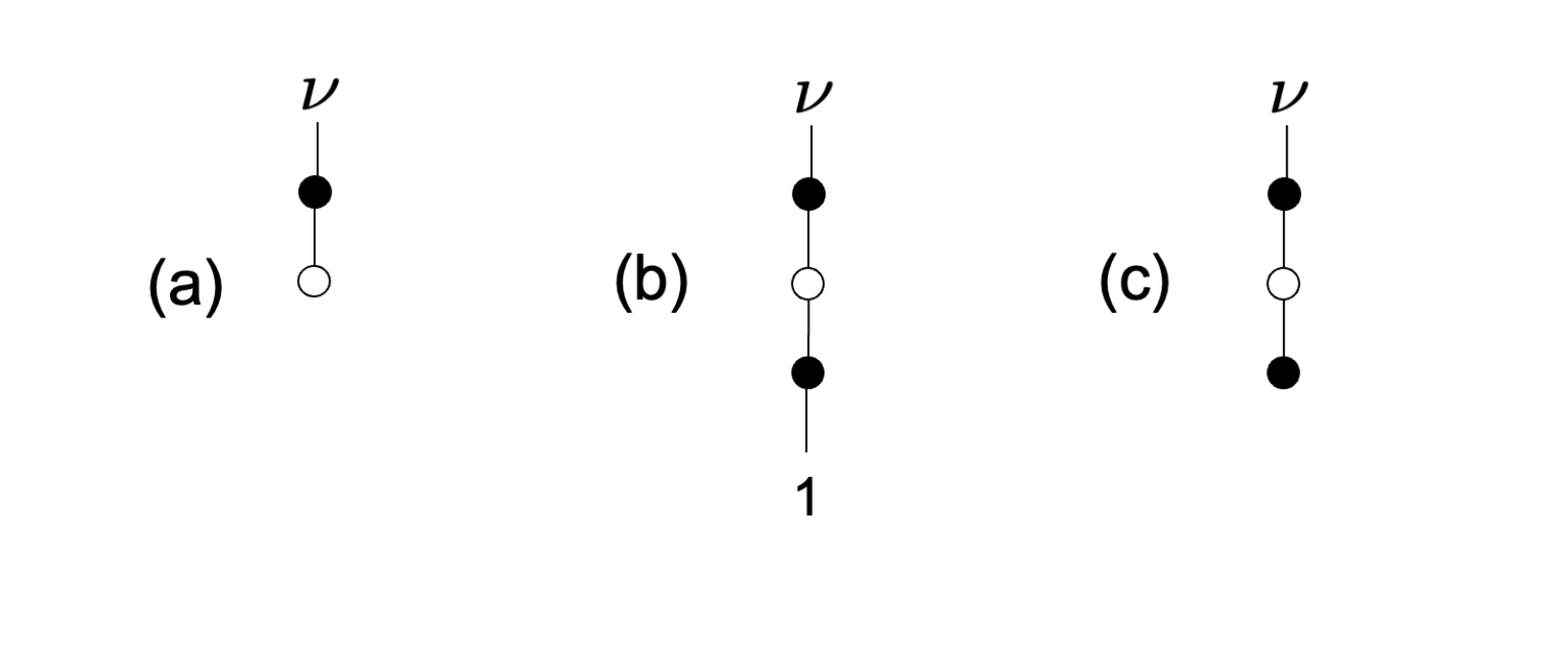}
\caption{\label{label}  Initial trees in Lemma 2.6 for (a) $\nu = 3n+2$, (b) $\nu = 3n+3$, (c) $\nu = 3n+4$.  }
\end{figure}
      \bigskip\par 
      
          \begin{lem} The series ${\mathcal{G}}^{(t)}_{d,\nu}(w), t=4,5,6$, can be extended to a larger series of polynomials, also denoted by ${\mathcal{G}}^{(t)}_{d,\nu}(w)$, having the following degrees $d=d(\nu, l, r, i)$
       \bigskip\par          
  (a)  $\nu= 3n+2$,  $n \in {\Bbb{Z}}^{\geq 1}$: $d(\nu, l, r, i)=(\nu+1)(l\nu + r+1) + 3i$,  $l \in {\Bbb{Z}}^{\geq 0}$,  $r=0,1,\dots 3n+1$, $i=0,1,\dots n$.
        \bigskip\par 
  (b) $\nu= 3n+3$,  $n \in {\Bbb{Z}}^{\geq 0}$:
  
    \hspace{0.5cm} (b1) $n=0$, $l \in {\Bbb{Z}}^{\geq 0}$, $r \in \{1,2,3\}$: $d(3, l, r, i)= 3(4l+r+1)+3i$, where $i=0$ if $r \in \{1,2\}$ and $i=0,1$, if $r=3$,
  
   \hspace{0.5cm}  (b2) $n \in {\Bbb{Z}}^{\geq 1}$, $l \in {\Bbb{Z}}^{\geq 0}$, $r=0, \dots \nu$: $d(\nu, l, r, i)=\nu(l(\nu +1)+r+1)+ 3i$, with  $i=1,\dots n$ if $r=0$;  $i=0,1,\dots, n$ if $r \in \{1, \dots , \nu-1\}$ and $i=0,1,\dots, n+1$ if $r= \nu$.
      \bigskip\par 
  (c) $\nu= 3n+4$,  $n \in {\Bbb{Z}}^{\geq 0}$, $l_{k}=3k+b, b\in \{ 0, 1, 2 \}$, $k \in {\Bbb{Z}}^{\geq 0}$: 
  
  \hspace{0.5cm}   (c1)  $l_{k}=0, r=-1$:  $d(\nu, l_{k}, r, i)= \nu +2+3i$, $i=0,1,\dots 2n+2$,
  
   \hspace{0.5cm}  (c2) $l_{k}  \in {\Bbb{Z}}^{\geq 0}$, $r=0,1,\dots n$ if $b\in \{ 0, 1 \}$, $r=-1,0,\dots n$ if $b=2$: $d(\nu, l_{k}, r, i)=(\nu +1)(l_{k} (\nu-1)+3(r+\lfloor\frac{l_{k}+1}{3}\rfloor +1))+ 3i$, $i=0,1,\dots \nu$.
          
  \end{lem}
           \begin{proof}     
               We denote  $C_{d}= \Big \lfloor \frac{d}{\nu+1} \Big \rfloor$, $D_{d}=\Big \lfloor \frac{d-1}{\nu} \Big \rfloor$, for a given value of $\nu$, and 
$U_{X}$, $X=C, D$ are the sets where $X_{d+3}- X_{d}=1$ if $i \in U_{X}$  and  $X_{d+3}- X_{d}=0$ if $i \notin U_{X}$. 

(a) For the cases  $\nu= 3n+2$ in Lemma 2.5 with $n \in {\Bbb{Z}}^{\geq 1}$ (the interpolation is not necessary when $n=0$) we consider $d(\nu, l, r, i)=(\nu+1)(l\nu + r+1) + 3i$, $l \in {\Bbb{Z}}^{\geq 0}$, $r=0,1,\dots 3n+1$, $i=0,1,\dots, n$, which satisfies $d(\nu, l+1, r, i)= d(\nu, l, r, i)+ \nu(\nu+1)$. The initial tree, corresponding to $l=r=i=0$, is represented in Fig.10(a). In this case $U_{C}=\{n\}, U_{D}=\{n-\Big \lfloor\frac{r+1}{3} \Big \rfloor \}$ if $r=0, 1, \dots \nu -2$ and $U_{D}=\{0, n\}$ if $r = \nu -1$. We also have $X_{d+3+\nu(\nu+1)}- X_{d+\nu(\nu+1)} = X_{d+3}- X_{d}, X=C, D$. 

 \bigskip\par
    
        (b)  For $\nu= 3n+3$, $n \in {\Bbb{Z}}^{\geq 0}$  the interpolated degrees $d(\nu, l, r, i)$ are
      \bigskip\par             
              \hspace{0.5cm}      (1)  $n=0$, $l \in {\Bbb{Z}}^{\geq 0}$, $r \in \{1,2\}$: $d(3, l, r, i)= 3(4l+r+1)$,
                 
              \hspace{0.5cm}      (2)  $n=0$, $l \in {\Bbb{Z}}^{\geq 0}$, $r=3$: $d(3, l, 3, i)= 3(4l+r+1)+3i$,  $i=0,1$, 
                 
                \hspace{0.5cm}     (3) $n \in {\Bbb{Z}}^{\geq 1}$, $l \in {\Bbb{Z}}^{\geq 0}$, $r=0$: $d(\nu, l, 0, i)=\nu(l(\nu +1)+1)+ 3i$, $i=1,\dots n$,
                  
                 \hspace{0.5cm}     (4) $n \in {\Bbb{Z}}^{\geq 1}$, $l \in {\Bbb{Z}}^{\geq 0}$, $r \in \{1, \dots , \nu-1\}$: $d(\nu, l, r, i)=\nu(l(\nu +1)+r+1)+ 3i$, $i=0,1,\dots, n$, 
                   
                   \hspace{0.5cm}    (5) $n \in {\Bbb{Z}}^{\geq 1}$, $l \in {\Bbb{Z}}^{\geq 0}$, $r= \nu$: $d(\nu, l, \nu, i)=\nu(l+1)(\nu +1)+ 3i$, $i=0,1,\dots, n+1$,

       \bigskip\par            
  \noindent               
 with the property $d(\nu, l+1, r, i)=d(\nu, l, r, i)+ \nu(\nu+1)$, $r \geq 0$.  The tree with the minimal degree $\nu+3$ ($d(3, 0, 1, 0)$ for $n=0$ and $d(\nu, 0, 0, 1)$ for  $n \geq 1$) is represented in Fig.10(b). The sets $U_{X}$ are 
      \bigskip\par 
    (1)  $n=0, l \in {\Bbb{Z}}^{\geq 0}$: 
    
      \hspace{0.5cm} (1.1) $r=1, 2$: $U_{C}= U_{D}=\{ 0 \}$,
    
       \hspace{0.5cm} (1.2) $r=3$: $U_{C}=\{ 1 \}, U_{D}=\{ 0,1 \}$,
    
     (2)  $n \in {\Bbb{Z}}^{\geq 1}, l \in {\Bbb{Z}}^{\geq 0}$: 
     
       \hspace{0.5cm} (2.1) $r=0$: $U_{C}= U_{D}=\emptyset$,
    
   \hspace{0.5cm} (2.2) $r=1, 2$: $U_{C}= U_{D}=\{0\}$, 
 
   \hspace{0.5cm} (2.3) $r=3, \dots \nu-1$: $U_{C}=\{ \lfloor \frac{r}{3}  \rfloor\}$,  $U_{D}=\{0\}$,
 
  \hspace{0.5cm}  (2.4) $r=\nu$: $U_{C}=\{ n+1 \}$,  $U_{D}=\{0, n+1\}$.
    \bigskip\par 
In this case $X_{d+3+ \nu(\nu+1)}- X_{d+ \nu(\nu+1)} = X_{d+3}- X_{d}, X=C, D$. 

  \bigskip\par
(c) For $\nu= 3n+4$ the interpolated degrees $d(\nu, l_{k}, r, i)$, $n \in {\Bbb{Z}}^{\geq 0}$, $l_{k}=3k+b, b\in \{ 0, 1, 2 \}$, $k \in {\Bbb{Z}}^{\geq 0}$ are 
      \bigskip\par 
   \hspace{0.5cm}  (1) $l_{k}=0$, $r=-1$: $d(\nu, 0, -1, i)=\nu+2+3i$,  $i=0,\dots 2n+2$,
  
    \hspace{0.5cm}  (2) $l_{k}=3k+b$, $b\in \{ 0, 1 \}$: $d(\nu, l_{k}, r, i)=(\nu +1)(l_{k} (\nu-1)+3(r+\lfloor\frac{l_{k}+1}{3}\rfloor +1))+ 3i$, $r = 0, \dots, n$, $i=0,\dots \nu$,
   
     \hspace{0.5cm}   (3) $l_{k}=3k+b$, $b=2$: $d(\nu, l_{k}, r, i)=(\nu +1)(l_{k} (\nu-1)+3(r+\lfloor\frac{l_{k}+1}{3}\rfloor +1))+ 3i$, $r = -1, 0, \dots, n$, $i=0,\dots \nu$,

       \bigskip\par 
  \noindent            
 with  $d(\nu, l_{k+1}, r, i)=d(\nu, l_{k}, r, i)+ 3\nu(\nu+1)$ and $X_{d+3+3 \nu(\nu+1)}- X_{d+3 \nu(\nu+1)} = X_{d+3}- X_{d}, X=C, D$, when $d \geq 3(\nu+1)$. The initial tree, with $d(\nu, 0, -1, 0)=\nu+2$, can be seen in Fig.10(c). In this case the sets $U_{X}$ are  
       \bigskip\par 
 (1) $n \in {\Bbb{Z}}^{\geq 0}$, $ l_{k}=0, r=-1$:   $U_{C}=\{n+1, 2n+2\}$,  $U_{D}=\{n, 2n+2\}$ 
 
 (2) $n=0$,  $l_{k}=3k+b, b\in \{ 0, 1, 2 \}$, $k \in {\Bbb{Z}}^{\geq 0}$: 
 
  \hspace{0.5cm}   (2.1) $l_{k}=3k+b, b\in \{ 0, 1\}$, $r=0$: $U_{C}=\{1, 3, 4\}$, $U_{D}=\{0, b+1, 3, 4\}$ 
  
    \hspace{0.5cm}  (2.2) $l_{k}=3k+2$: $U_{C}=\{1, 3, 4\}$ and $U_{D}=\{1, 2, 3\}$ if $r = -1$ and $U_{D}=\{0, 1, 2, 4\}$ if $r = 0$
 
   (3) $n \in {\Bbb{Z}}^{\geq 1}$, $l_{k}=3k+b, b\in \{ 0, 1, 2 \}$, $k \in {\Bbb{Z}}^{\geq 0}$: $U_{C}=\{n+1, 2n+3, 3n+4\}$ and 
 
  \hspace{0.5cm}  (3.1) $l_{k}=3k+b, b\in \{ 0, 1\}$: $U_{D}=\{n-r, 2n-r+b+1, 3n-r+3\}$ if $r = 0, \dots n -1$ and 
 
 $U_{D}=\{0, n+b+1, 2n+3, 3n+4\}$ if $r = n$,
 
   \hspace{0.5cm} (3.2) $l_{k}=3k+2$:  $U_{D}=\{n-r, 2n-r+1, 3n-r+2\}$ if $r = -1, 0, \dots n -1$ and 
 
 $U_{D}=\{0, n+1, 2n+2, 3n+4\}$ if $r = n$.

    \bigskip\par 
 We now apply a series of transformations $w= \alpha$, $\beta$, $\bar{\alpha}$, $\bar{\beta}$ to $\mathcal{T}_{\nu}(d, N_{-1}, N_{1})$ giving $\mathcal{T}_{\nu}(d+3, N_{-1}+a, N_{1}+b), a,b \in \{-1,0,1\}$. The transformation $\alpha$ consists in adding black vertices to a white one in such a way that it is transformed into one with $\nu+1$ edges, hence $a=0$, $b=1$. With $\beta$ a black vertex is converted into one with $\nu+1$ adjacent edges with a resulting tree $\mathcal{T}_{\nu}(d+3, N_{-1}+1, N_{1}-1)$. The transformation $\bar{\alpha}$ is like $\alpha$ but the transformed vertex has less than $\nu+1$ adjacent edges ($a=0$, $b=0$) and  $\bar{\beta}$ is like $\beta$ but without subtracting edges from a white vertex ($a=1$, $b=0$). By employing the notation  $Y= (C_{d+3}- C_{d}, D_{d+3}- D_{d}; w)$, which means that $w$ is applied when the first two terms of the triple take the indicated values, we have $Y \in$ $\{(0,1; \alpha)$,  $(1,0; \beta)$, $(0, 0; \bar{\alpha})$,  $(1, 1; \bar{\beta})\}$. We analyse the words representing the sequence of trees, which can be obtained from the sets $U_{X}$ described above.
  
      \bigskip\par 
 
(a) $\nu= 3n+2$. If we start with the tree for $n=1, l=1$ in Lemma 2.5 then a sequence of allowed words is $W^{c}$, $W^{c}\bar{\alpha}$, $W^{c}\bar{\alpha}\bar{\beta}$, $\dots$, $W^{c+1}$, $W^{c+1}\bar{\alpha}$, $\dots $, $c \in  {\Bbb{Z}}^{\geq 1}$, where $W=(\bar{\alpha}\bar{\beta})^{2}(\alpha\beta)^{2}\alpha\bar{\beta}$ describes the chain of transformations from $d=36$ to $d=66$. The words appear repeated in cycles, with one cycle from $l$ to $l+1$. In this Lemma we also incorporate $l=0$, which for $n=1$ corresponds to $d=6$, and then the word representing the transformations to the tree with $d=36$ is also $W$. In what follows $W$ represents the complete word within a cycle, and the allowed words are all its prefixes. If we use the product notation for concatenation then for $n \geq 2$  the words within a cycle are the prefixes of 

  \begin{equation} 
 W=\bar{Q}_{n}^{2}\left(\prod^{n-1} _{p=1}P_{n, p-1}^{3}\right) (\alpha Q_{n-1})^{2} \alpha \bar{Q}_{n-1}
           \end{equation} 
with $P_{n,p}=\bar{\alpha}^{n-p-1}\alpha \bar{\alpha}^{p} \beta, Q_{n}= \bar{\alpha}^{n}\beta, \bar{Q}_{n}= \bar{\alpha}^{n}\bar{\beta}$. The trees associated to the first two cycles are represented in Fig.11(a).
\par
                  \begin{figure}[h]
 \includegraphics[width=26pc]{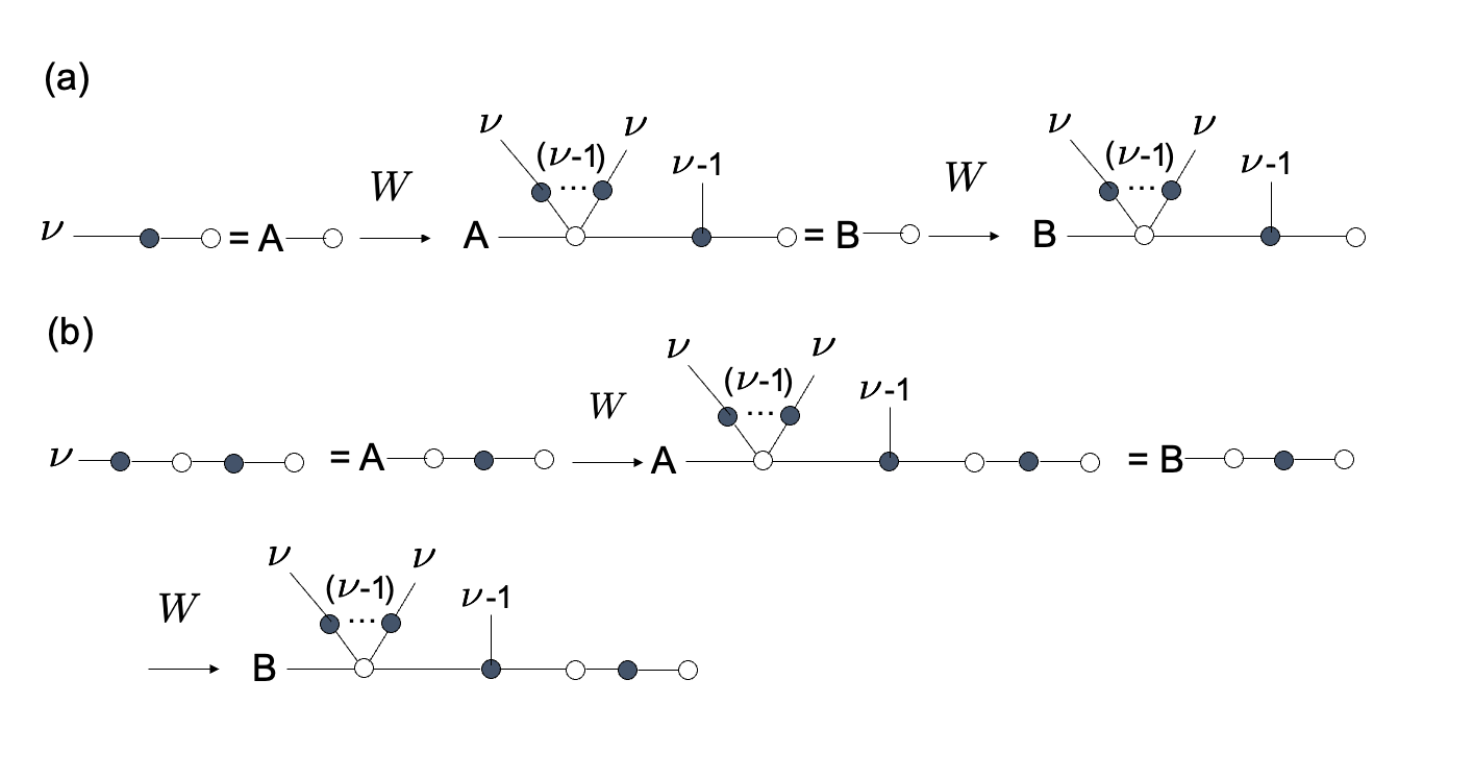}
\caption{\label{label} Trees for the first two cycles $W^{c}, c=1,2$, in Lemma 2.6 for (a) $\nu= 3n+2$, (b) $\nu= 3n+3$.}
\end{figure}  

     \bigskip\par 
 
(b) $\nu= 3n+3$. The cycle in this case is from $l$ to $l+1$ and its word $W$ is 
      \bigskip\par 
  \hspace{0.5cm} (b1) $n =0$, $l \in  {\Bbb{Z}}^{\geq 0}$:

      \begin{equation} 
W=  \bar{Q}_{0}^{2}  \alpha \bar{Q}_{0}.
          \end{equation}  

  \hspace{0.5cm} (b2) $n \in  {\Bbb{Z}}^{\geq 1}$, $l \in  {\Bbb{Z}}^{\geq 0}$:
      \begin{equation} 
W=  \bar{Q}_{n}^{2} \bar\alpha^{n} \left( \prod^{n-1} _{p=0}( \alpha {Q}_{p } \bar\alpha^{n-p -1})^{3}\right) \alpha  \bar{Q}_{n}.
          \end{equation}  

 The possible words are $W^{c}$, $c \in  {\Bbb{Z}}^{\geq 1}$ and its prefixes. The cases $c=1,2$ are represented in Fig.11(b).
      \bigskip\par 
      
                        \begin{figure}[h]
 \includegraphics[width=20pc]{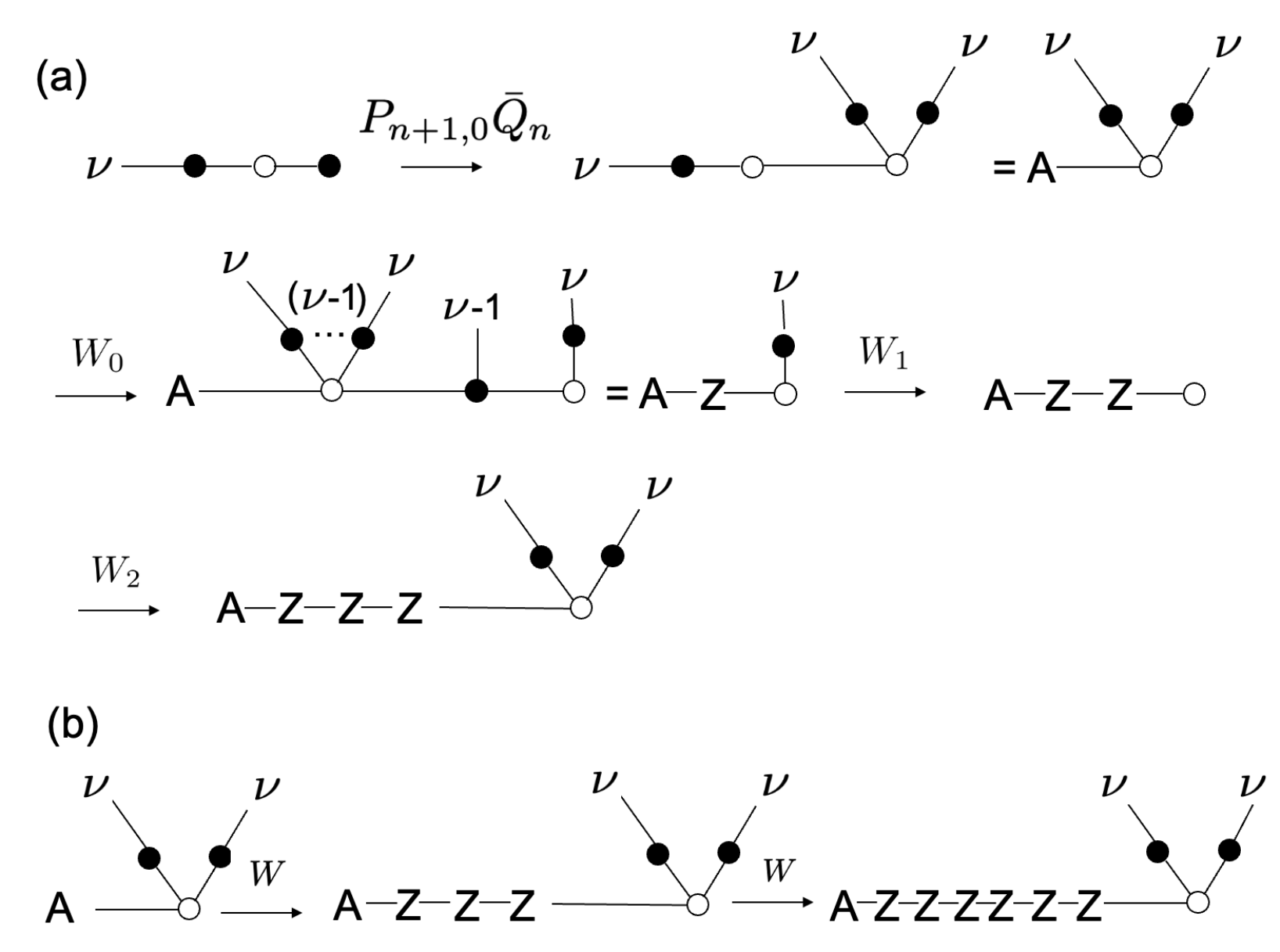}
\caption{\label{label} Trees in Lemma 2.6 for $\nu= 3n+4$.  (a) The application of the transformations associated to $P_{n+1,0} \bar{Q}_{n}$ and then to $W_{0}W_{1}W_{2}$. (b) The first two cycles $W^{c}, c=1,2$ with $W=W_{0}W_{1}W_{2}$.}
\end{figure} 

(c) $\nu= 3n+4$. The  words  are  $P_{n+1,0} \bar{Q}_{n} W^{c}$, $W=\prod^{2}_{b = 0} W_{b}$, $c \in  {\Bbb{Z}}^{\geq 1}$ and its prefixes. The cycle occurs from $l_{k}$ to $l_{k+1}$  and for each $l_{k}=3k+b, k \in  {\Bbb{Z}}^{\geq 0}, b\in \{ 0, 1, 2 \}$ the words only depend on $b$.
      \bigskip\par 
  \hspace{0.5cm} (c1) $n =0$, $l \in  {\Bbb{Z}}^{\geq 0}$:

   \begin{equation} 
  W_{0}=\alpha\bar{\beta}\bar{\alpha}\bar{\beta}^2, W_{1}=\alpha\beta\alpha\bar{\beta}^2, W_{2}=\bar{\alpha}\bar{\beta}\alpha\bar{\beta}\beta\alpha\bar{\beta}\alpha\beta\bar{\beta}.
           \end{equation} 

  \hspace{0.5cm} (c2) $n \in  {\Bbb{Z}}^{\geq 1}$, $l \in  {\Bbb{Z}}^{\geq 0}$:
    \begin{equation} 
  W_{0}=\left(\prod^{n-1} _{p=0} P_{n+1,p}P_{n+1,p+1}P_{n,p}\right)\alpha\bar{Q}_{n}\bar{Q}_{n+1} \bar{Q}_{n},  
           \end{equation} 
  \begin{equation} 
W_{1}=\left(\prod^{n-1} _{p=0} (P_{n+1,p})^{2}P_{n,p}\right)\alpha Q_{n}  \alpha  \bar{Q}_{n}^{2},
         \end{equation}  
  
      \begin{equation} 
    W_{2}= \bar{Q}_{n+1} P_{n+1,0}P_{n,0}\left(\prod^{n-2} _{p=0}P_{n+1,p}P_{n+1,p+1}P_{n,p+1}\right)\bar{\alpha} \alpha Q_{n-1} \alpha 
    \bar{Q}_{n}Q_{n}\alpha\bar{Q}_{n}P_{n+1,0} \bar{Q}_{n}.
             \end{equation}  
 In Fig.12(a) we have represented the tree obtained after applying the transformations corresponding to $P_{n+1,0} \bar{Q}_{n}$ and then to $W=W_{0}W_{1}W_{2}$.  The trees associated to $W^{c}, c=1,2$ are represented in Fig.12(b).
  \end{proof}
  
 We are studying polynomials having $N_{1}$ as big as possible. In the interpolation process, giving degree $d+3i$ polynomials, $N_{1}$ increases for some $i$ in comparison to its value for the degree $d$ polynomial, which is $N_{1}=l$. This is treated in the Appendix. 
 \par         
The same method can be applied to get polynomials with all degrees by interpolating between two consecutive degrees $d$ and $d+3$ in the families obtained in Lemma 2.6.  We use the same notation for the transformations, having in mind that now $w= \alpha$, $\beta$, $\bar{\alpha}$, $\bar{\beta}$ transforms $\mathcal{T}_{\nu}(d, N_{-1}, N_{1})$ into $\mathcal{T}_{\nu}(d+1, N_{-1}+a, N_{1}+b), a,b \in \{-1,0,1\}$. Also $Y= (C_{d+1}- C_{d}, D_{d+1}- D_{d}; w)$, with $Y \in$ $\{(0,1; \alpha)$,  $(1,0; \beta)$, $(0, 0; \bar{\alpha})$,  $(1, 1; \bar{\beta})\}$. The new words can be obtained by replacements of type $x \mapsto uvw$, $x,u,v,w \in \{\alpha, \beta, \bar{\alpha}, \bar{\beta} \}$, in the words given by Eqs.(2.7)-(2.13). For example, the word within a cycle for  $\nu=3n+2, n=2$ in Eq. (2.7) is  $W=W_{1}W_{2}W_{3}W_{4}$ with $W_{1}= (\bar{\alpha}^2\bar{\beta})^2$,  $W_{2}= (\bar{\alpha}\alpha\beta)^3$, $W_{3}= (\alpha\bar{\alpha}\beta)^2$, $W_{4}= \alpha\bar{\alpha}\bar{\beta}$. The replacements for $\bar{\alpha}$ and $\beta$ are always $\bar{\alpha} \mapsto \bar{\alpha}^3$, $\beta \mapsto \bar{\alpha}^2\beta$. The letters $\alpha, \bar{\beta}$ are replaced in the following ordered sequence as they appear from left to right in $W$: $\alpha \mapsto \bar{\alpha}^2\alpha$, $\bar{\alpha}\alpha\bar{\alpha},$ $\alpha\bar{\alpha}^2$,  $\bar{\beta} \mapsto \bar{\alpha}\alpha\beta$, $\alpha\bar{\alpha}\beta$, $\bar{\alpha}^2\bar{\beta}$ (the sequence of replacements is applied twice for $\alpha$). The word corresponding to all degrees in this case is $\bar{W}=\bar{W_{1}}\bar{W_{2}}\bar{W_{3}}\bar{W_{4}}$  with    

$$\bar{W_{1}}= \bar{\alpha}^7\alpha\beta\bar{\alpha}^6\alpha\bar{\alpha}\beta,  \bar{W_{2}}= \bar{\alpha}^5\alpha\bar{\alpha}^2\beta\bar{\alpha}^4\alpha\bar{\alpha}^3\beta\bar{\alpha}^3\alpha\bar{\alpha}^4\beta, \bar{W_{3}}= \bar{\alpha}^2\alpha \bar{\alpha}^5\beta\bar{\alpha}\alpha\bar{\alpha}^6\beta, 
\bar{W_{4}}= \alpha\bar{\alpha}^7\bar{\beta}.$$
               
 \section{The singular surfaces}
 \bigskip\par
We denote the number of singularities of type $A_{\nu}$ in a surface $\mathcal{S}$  by $\mathcal{N}(\mathcal{S},A_{\nu})$  and the maximal number of $A_{\nu}$ complex singularities for a degree $d$ surface by $\mu_{A_{\nu}}(d)$. The affine equations of the surfaces considered in this section are obtained using the polynomials  ${\mathcal{G}}^{(t)}$ discussed in Lemmas 2.3-2.6, and polynomials related to the family 
    \begin{equation} 
\hat{J}_{d,\tau}(x, y):=\lambda_{d, \tau}  \prod_{\mu} L_{d, \tau, \mu}\left(x, y\right)
   \end{equation} 
   where  $L_{d,\tau,\mu}(x,y):=y+({\rm cos}2\varphi-x){\rm tan}\varphi+{\rm sin}2\varphi=0, \varphi=\frac{(6\mu-1)\pi }{6d}-\frac{\tau}{d}, \mu=- \lfloor \frac{d-2}{2} \rfloor,- \lfloor \frac{d-2}{2} \rfloor+1,\dots , \lfloor \frac{d+1}{2} \rfloor$ with $(x,y)\in {\Bbb{R}}^2$ and $\tau \in \Bbb{R}$. The parameters are $\lambda_{d,\tau} =(-1)^{m}2 d$ if $\tau = (6m-3d-1)\frac{\pi}{6}$ (in this case the line  $L_{d,\tau,\mu}(x,y)=0$ parallel to the $y$-axis is interpreted as the line $x+1=0$) and  $\lambda_{d,\tau}  = 2{\rm cos}(\tau+\frac{d\pi}{2}+\frac{2\pi}{3})$ if $\tau \neq (6m-3d-1)\frac{\pi}{6}$, $m \in \Bbb{Z}$.
  \par
The following result characterises the critical points and critical values of  $\hat{J}_{d,0}(x, y)$. They have multiplicity one  (\cite{esc16}, Lemma 1).

    \begin{lem}  The real polynomial $\hat{J}_{d,0}(x, y)$ has ${d \choose 2}$ critical points with critical value $0$. The number of points with critical value $8$ is $\frac{d(d-3)}{6}$ if $d=0$ mod 3, and $\frac{(d-1)(d-2)}{6}$ otherwise. The number of critical points with critical value $-1$ is $\frac{d^{2}}{3}-d+1$ for $d=0$ mod 3, and $\frac{(d-1)(d-2)}{3}$ otherwise. 
     \end{lem}

  \par
  In (\cite{esc18}, Lemma 4.2) we have shown that
     \begin{equation} 
      {\mathcal{J}}_{d}(x, y):= \hat{J}_{d, 0}\left(x, \frac{y}{ \sqrt{3}}\right)
           \end{equation}      
are defined over ${\Bbb{Q}}$ and in (\cite{esc18}, Theorem 4.3) the following. 

\begin{thm}
 \par
1. For $d=3m+r, r \in\{0,1,2\}, m \in {\Bbb{Z}}^{+}$, the projective surfaces defined over ${\Bbb{Q}}$ with affine equations 

     \begin{equation} 
      {\mathcal{J}}_{d}(x, y)+\frac{1}{4}(3-{\mathcal{J}}_{d}(2z+1, 0))=0 
       \end{equation}     
         have 
            \begin{equation}  
        {d \choose 2} \Big  \lfloor \frac{d}{2} \Big  \rfloor+\left(\frac{d^{2}-\lceil\frac{r}{2}\rceil}{3}-d+1\right) \Big \lfloor \frac{d-1}{2} \Big \rfloor
       \end{equation}         
        real nodes and no other singularities. 
        \par
2. The number of real nodes of the projective threefolds defined over ${\Bbb{Q}}$ with affine equations 
   \begin{equation} 
   {\mathcal{J}}_{d}(x, y)-{\mathcal{J}}_{d}(z, w)=0
          \end{equation}     
    is
            \begin{equation}  
{d \choose 2}^2+\left(\frac{d^{2}-3d+2\lceil\frac{r}{2}\rceil}{6}\right)^2+\left(\frac{d^{2}-\lceil\frac{r}{2}\rceil}{3}-d+1\right)^2
       \end{equation}   
and the threefolds have no other singularities.
\end{thm}        
 \bigskip\par
If $\mathcal{N}$ is the number of nodes given by Eq. (3.4)  then 
            \begin{equation} 
\mu_{A_{1}}(d) \geq \mathcal{N}. 
      \end{equation}   
Also Eq. (3.6) gives a lower bound for nodal threefolds. We now study surfaces whose affine equations are, as in Eq. (3.3), of the form ${\mathcal{J}}+{\mathcal{P}}=0$, where ${\mathcal{P}}$ has critical values $0,1$. We use the polynomial ${\mathcal{U}}^{(t)}_{d,\nu,\epsilon}(w)=\frac{1}{2}({\mathcal{G}}^{(t)}_{d,\nu,\epsilon}(w)+1)$.
 \bigskip\par
      \begin{thm} 
The degree $d=3q$ surfaces $\mathcal{S_{J}}$ with affine equations  
           \begin{equation} 
{\mathcal{J}}_{d}(u, v)+{\mathcal{U}}^{(t)}_{d,\nu,\epsilon}(w)=0
      \end{equation} 
 where $(u, v, w) \in {\Bbb{C}^3}$,  have 
 \par       
 \begin{equation} 
\mathcal{N}(\mathcal{S_{J}}, A_{\nu})= \frac{d(d-1)}{2}  \Big \lfloor \frac{d}{\nu+1} \Big \rfloor+\left(\frac{d(d-3)}{3}+1\right)\left(\Big \lfloor \frac{d-1}{\nu}   \Big \rfloor-    \Big \lfloor \frac{d}{\nu+1}   \Big \rfloor\right)
 \end{equation}  
 for $d>\nu\geq 2$.

 \end{thm}
 
 \begin{proof}
The number of singularities of the surfaces described by Eq. (3.8) can be obtained from 
$$\mathcal{N}(\mathcal{S_{J}},A_{\nu})=N_{0}(\mathcal{J},1)N_{-1}(\mathcal{G}, \nu)+N_{-1}(\mathcal{J},1)N_{1}(\mathcal{G}, \nu).$$
For $d=3q$ the number of critical points of ${\mathcal{J}}_{d}$ is (Lemma 3.1)

$$N_{0}(\mathcal{J},1)=\frac{d(d-1)}{2}, N_{8}(\mathcal{J},1)=\frac{d(d-3)}{6}, N_{-1}(\mathcal{J},1)=\frac{d(d-3)}{3}+1.$$  

By taking into account the results of Lemmas 2.3 and 2.4 we see  that $\mathcal{N}(\mathcal{S_{J}},A_{\nu})=\frac{d(d-1)}{2} \Big \lfloor \frac{d}{\nu+1} \Big \rfloor+\frac{d(d-3)}{3}+1$, when  $t=1,2,3, \nu > 2$, because the polynomials satisfy condition $(E)$ with $N_{1}(\mathcal{G}, \nu)=1$. For the case $t=1,n=0,m=1$ we have $N_{-1}(\mathcal{G}, 2)=3+h$ and $N_{1}(\mathcal{G}, 2)=1+\lfloor\frac{h}{2}\rfloor$, therefore
  $\mathcal{N}(\mathcal{S_{J}},A_{2})=\frac{d(d-1)(3+h)}{2}+(\frac{d(d-3)}{3}+1)(1+\lfloor\frac{h}{2}\rfloor)$ with $d=3(h+3)$.
    \par
For $t=4,5,6$ the polynomials satisfy condition $(E)$ with $N_{1}(\mathcal{G}, \nu)=s=\Big \lfloor \frac{d-1}{\nu}   \Big \rfloor-    \Big \lfloor \frac{d}{\nu+1}   \Big \rfloor  \in {\Bbb{Z}}^{\geq 1}$ (Lemmas 2.5 and 2.6), hence $\mathcal{N}(\mathcal{S_{J}}, A_{\nu})= \frac{d(d-1)}{2}  \Big \lfloor \frac{d}{\nu+1} \Big \rfloor+(\frac{d(d-3)}{3}+1)s$. The series corresponding to $t=4, \nu=2$ has degrees $d=3(2l+r+1)$ with $r=0,1, l\in {\Bbb{Z}}^{\geq 1}$ and the same number of singularities as the case $t=1,n=0,m=1$.
 
    \end{proof}
    \newtheorem*{remark}{Remark}
    \begin{remark}
  The projective surfaces have no singularities at infinity as follows from the fact that ${\mathcal{J}}_{d}(u, v)$ has $(d-1)^2$ finite critical points  and Bézout´s theorem \cite{esc13,esc16,esc18}. However the surfaces may have other singularity types. In Lemmas 1, 2 and 3 in \cite{esc26} it is shown that  $N_{1}(\mathcal{B}, \nu)=1$, and $N_{-1}(\mathcal{B}, \nu)=\Big \lfloor \frac{d_{0}}{\nu+1} \Big \rfloor$, $N_{-1}(\mathcal{B}, \epsilon)=1$, $N_{1}(\mathcal{B}, \epsilon)=0$ \cite{esc26}, hence  $\mathcal{N}(\mathcal{S_{J}},A_{\epsilon})=\frac{d_{0}(d_{0}-1)}{2}$ when $\epsilon \neq 0$. Other examples include ${\mathcal{G}}^{(t)}_{d,\nu,\epsilon}(w), t=1, 3$ where $N_{1}(\mathcal{G}, 1)=1$ if $h$ is odd, hence $\mathcal{N}(\mathcal{S_{J}},A_{1})=\frac{d(d-3)}{3}+1$ in that case.
        \end{remark}
 \bigskip\par

    \begin{remark}
The surfaces $\mathcal{S}_{F}$ studied in  \cite{lab06} have affine equations $F_{d}^{\bf{A}_{2}}(u, v)+M^{j}_{d}(w)=0$, where $F_{d}^{\bf{A}_{2}}(u,v)$ denotes the folding polynomial associated to the root lattice $\bf{A}_{2}$ obtained in \cite{wit88, hof88}, and $M^{j}_{d}(w)$ are degree $d$ Bely polynomials with $\lfloor \frac{d}{j+1}\rfloor$ critical points of multiplicity $j$ with critical value $-1$ and $\lfloor \frac{d-1}{j}\rfloor-\lfloor \frac{d}{j+1}\rfloor$ critical points with critical value $+1$. According to Theorem 3.3, the surfaces $\mathcal{S_{J}}$ have $N_{1}(\mathcal{G}, \nu)$ more singularities than $\mathcal{S}_{F}$, namely $1+\Big \lfloor\frac{h}{2}\Big \rfloor$ more cusps for $t=1$, one more singularity of type $A_{\nu}$ when $t=1,2,3, \nu>2$, and $s\in {\Bbb{Z}}^{\geq 1}$ more singularities of type $A_{\nu}$ when $t=4,5, 6, \nu \geq 2$. If we compare Eq.3.9  with Eq.3.4 for $r=0$ we see that Eq.3.9 is also valid for nodes and we can give the following bound for $d=3q$ and $\nu \geq 1$:
     \begin{equation}  
\mu_{A_{\nu}}(d) \ge \left( \frac{d(d+3)}{6}-1\right)  \Big \lfloor \frac{d}{\nu+1} \Big \rfloor+\left(\frac{d(d-3)}{3}+1\right)\Big \lfloor \frac{d-1}{\nu}   \Big \rfloor .
     \end{equation} 
     As we have commented at the end of Section 2, the interpolation process can be extended to get polynomials satisfying condition $(E)$ for all $d$. For $d \neq 3q$ the surfaces, that we also denote by $\mathcal{S_{J}}$, with affine equations ${\mathcal{J}}_{d}(u, v)+{\mathcal{U}}^{(t)}_{d,\nu,\epsilon}(w)=0$ have  $\mathcal{N}(\mathcal{S_{J}}, A_{\nu})=\frac{d(d-1)}{2}  \Big \lfloor \frac{d}{\nu+1} \Big \rfloor+\frac{d(d-3)}{3}\left(\Big \lfloor \frac{d-1}{\nu}   \Big \rfloor-    \Big \lfloor \frac{d}{\nu+1}   \Big \rfloor\right)$ according to Lemma 3.1.
             \end{remark}
     \par
     \section{Explicit equations for surfaces over the rationals with complex and real singularities}
     
       Some polynomials can be computed for low degree using Groebner basis as in \cite{ lab06, esc14a, esc14b}. We study certain cases having explicit expressions in terms of classical Jacobi polynomials $P_{k}^{(l, m)}(z)$ \cite{abr70}, \cite{sze75} that can be obtained along the lines of \cite{esc26}, where we have analysed the three-parameter family of degree $d=a+(b-1)c$ polynomials 
   \begin{equation}    
     G_{a, b, c}(w)=w^{a}P_{b-1}^{(a/c,-b)}(1-2w)^{c}.
   \end{equation} 
  There is one zero of $\frac{dG_{a, b, c}(w)}{d w}$ with critical value $\zeta=0$ and order $\nu=a-1$, there are also $b-1$ zeros with critical value $\zeta=0$ and $\nu=c-1$, while the unique remaining root $w=1$ has $\nu=b-1$ with critical value $\zeta=1$. The trees for $G_{a, b, c}(w)$ are represented in \cite{esc26}, Fig. 4.  We consider polynomials in Lemma 2.6 that can be written in terms of $G_{a, b, c}(w)$.  
   
          \bigskip\par
      \begin{prop}
     For  $\nu \in \{ 3n+2, 3n+3, 3n+4 \}, n\in {\Bbb{Z}}^{\geq 0}$, $b \in \{2, ..., \nu+1\}$,  $a \in \{1, ..., \nu-b+1\}$ if $b \in \{2, ..., \nu\}$ and  $a \in \{1, ..., \nu\}$ if $b=\nu+1$, the degree $d=a+(b-1)(\nu+1)$ surfaces over ${\Bbb{Q}}$ with affine equations
          \begin{equation} 
       {\mathcal{J}}_{d}(u, v)+ w^{a}P_{b-1}^{(\frac{a}{\nu+1},-b)}(1-2w)^{\nu+1}=0
                 \end{equation} 
  \par\noindent     
     where $u, v, w \in {\Bbb{C}}$, have the following number of $A_{\nu}$ singularities:
        
       (1)  $\frac{d(d-1)(b-1)}{2}$ if  $b \in \{2,  ..., \nu\}$,
       
       (2) $\frac{d(d-1)(b-1)}{2} +\frac{d(d-3)}{3}+k$ if $b= \nu+1$, where $k=1$ if $a \in 3{\Bbb{Z}}$ for $\nu \in \{ 3n+2, 3n+3\}$ or $a+2 \in 3{\Bbb{Z}}$ for $\nu=3n+4$, and $k=0$ otherwise.
       \end{prop}
 \begin{proof}

       \bigskip\par    
We first interpolate in the families discussed in Lemma 2.6 to get trees with $l=0$ and all degrees. By comparing with the corresponding trees of $G_{a, b, c}(w)$ we see that we can give explicit expressions for the affine equations when $c=\nu +1$, $b \in \{2, ..., \nu+1\}$, $a \in \{1, ..., \nu-b+1\}$ if $b \in \{2, ..., \nu\}$ and  $a \in \{1, ..., \nu\}$ if $b=\nu+1$. The degrees are $d=a+(b-1)(\nu +1)$ with $ \Big \lfloor \frac{d}{\nu+1} \Big \rfloor = \Big \lfloor \frac{d-1}{\nu} \Big \rfloor = b-1$ if $b \in \{2, ..., \nu\}$ and  $ \Big \lfloor \frac{d}{\nu+1} \Big \rfloor = \Big \lfloor \frac{d-1}{\nu} \Big \rfloor -1 = b-1$ if $b= \nu+1$. According to Lemmas 2.6, 3.1 and the interpolations for $d \notin 3{\Bbb{Z}}$, the number of $A_{\nu}$ singularities is $\frac{d(d-1)(b-1)}{2}$ if  $b \in \{2, ..., \nu\}$. When $b= \nu+1$ we have $d \in 3{\Bbb{Z}}$ if $a \in 3{\Bbb{Z}}$ for $\nu \in \{ 3n+2, 3n+3\}$ and if $a+2 \in 3{\Bbb{Z}}$ for $\nu=3n+4$, hence for those cases the number of  $A_{\nu}$ singularities is $\frac{d(d-1)\nu}{2} +\frac{d(d-3)}{3}+1$. For  $d \notin 3{\Bbb{Z}}$ and $b= \nu+1$ the number of singularities is $\frac{d(d-1)\nu}{2} +\frac{d(d-3)}{3}$.
              \end{proof}
              
                 \begin{figure}[h]
 \includegraphics[width=26pc]{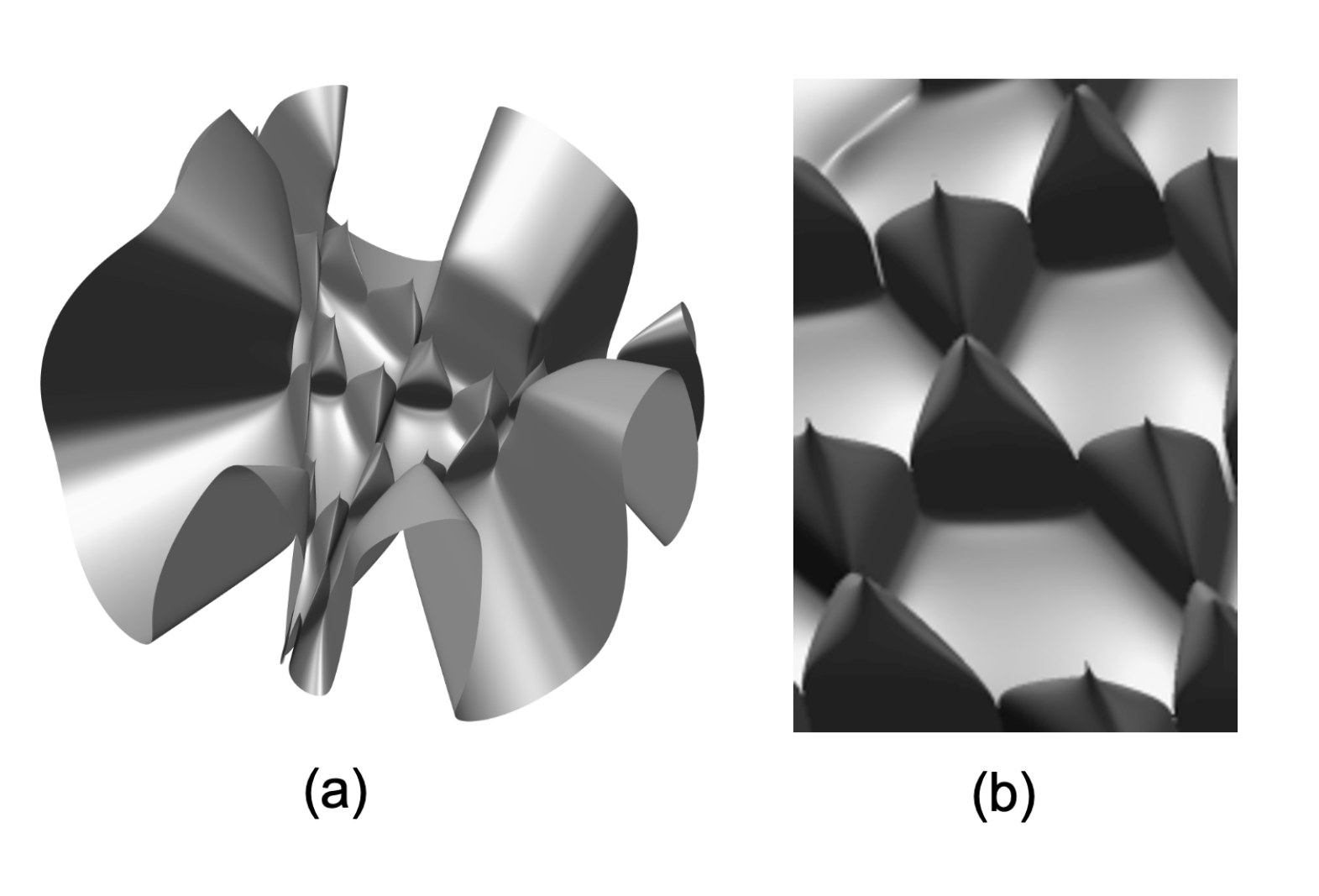}
\caption{\label{label} (a) Degree 9 surface over the rationals with 55 real $A_{4}$ singularities. 
(b) Fragment of a degree 21 surface over the rationals with 337 real $A_{10}$ singularities.}
\end{figure}  
\par
          \begin{figure}[h]
 \includegraphics[width=30pc]{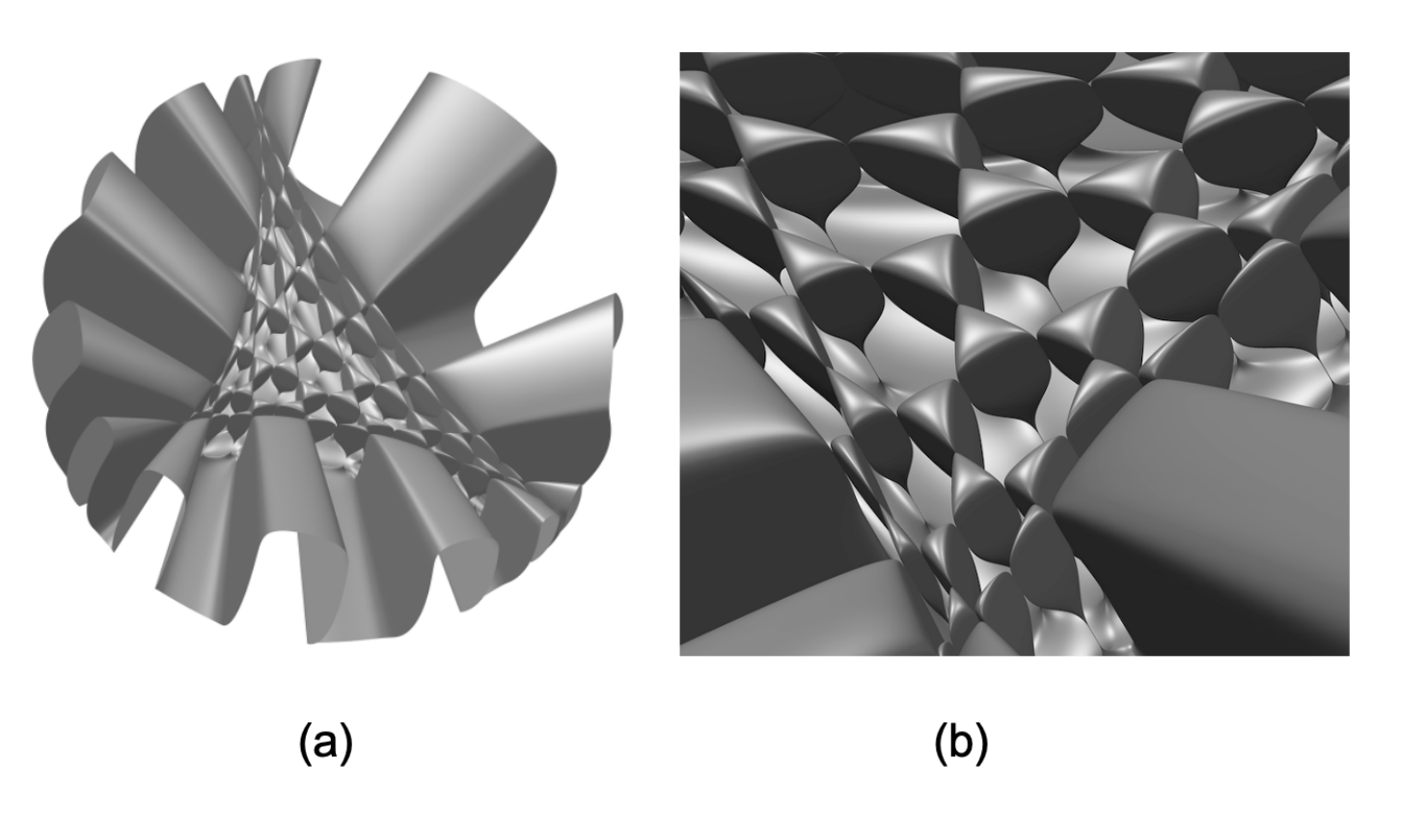}
\caption{\label{label} (a) Degree 15 surface over the rationals with 166 real $A_{7}$ singularities. (b) Fragment of the surface with a small deformation.}
\end{figure}

    \bigskip\par    
 For the polynomials in Lemma 2.6 corresponding to  $\nu=3n+4$, $d=2\nu+1$  we study the family
   \begin{equation} 
G_{\nu+1,\nu+1,1}(z)=z^{\nu+1}P_{\nu}^{(\nu+1, -\nu-1)}(1-2z).
     \end{equation}  
         
       \bigskip\par
      \begin{prop} 
The degree $d=6n+9$ surfaces over ${\Bbb{Q}}$, with $\nu=3n+4, n\in {\Bbb{Z}}^{\geq 0}$,  and affine equations  
           \begin{equation} 
{\mathcal{J}}_{d}(x, y)+\frac{(2\nu+1)!}{\nu!}\sum_{m=0}^{\nu}\frac{(-1)^{m}z^{m+\nu+1}}{(\nu-m)!m!(m+\nu+1)}=0
      \end{equation} 
 where $(x, y, z) \in {\Bbb{R}^3}$,  have $30 n^2+ 81 n+ 55$ real $A_{3n+4}$ singularities.

 \end{prop}
 \begin{proof}
 The critical points of $G_{\nu+1,\nu+1,1}(z)$ are $z=0$ with $\zeta=0$ and $z=1$ with $\zeta=1$, both with order $\nu$.  If we write the Jacobi polynomial in Eq.(4.3) in terms of the hypergeometric function $F(-\nu,\nu+1; \nu+2; z)$ \cite{abr70}, \cite{sze75}, we get 
           \begin{equation} 
        P_{\nu}^{(\nu+1, -\nu-1)}(1-2z)=\frac{(2\nu+1)!}{\nu!}\sum_{m=0}^{\nu}\frac{(-1)^{m}z^{m}}{(\nu-m)!m!(m+\nu+1)}.
             \end{equation}  
 
For  $d=2\nu+1$  we have $ \Big \lfloor \frac{d}{\nu+1} \Big \rfloor = 1$ and $\Big \lfloor \frac{d-1}{\nu} \Big \rfloor = 2$, hence the surfaces given by Eq. (4.4) have $\frac{d(d-1)}{2} +\frac{d(d-3)}{3}+1=30 n^2+ 81 n+ 55$ real $A_{3n+4}$ singularities.
 
    \end{proof}

Explicit expressions for the cases with degree $d=6n+9$ and $A_{3n+4}$ singularities were obtained in \cite{esc14a} for $d=9$ and in \cite{esc14b} for $d=15$ by using Groebner basis, but the resulting surfaces were not defined over ${\Bbb{Q}}$. In Fig.13  we can see parts of the degrees 9 and 21 surfaces given by Eq. (4.4) which are defined over the rationals. The degree 15 surface with $A_{7}$ singularities is represented in Fig.14(a). In order to show the neighbourhood of some singularities where the surface is very thin and not visible in Fig.14(a), we represent in Fig.14(b) a deformation of the surface obtained by adding a small negative number to the affine equation. The surfaces have been visualised with the {\it Surfer} tool  \cite{sur}. In Figs.13(a), 14(a) they are represented within an invisible sphere.

 \section{Appendix}
 
 In Lemma 2.6  the degree $d$ polynomial has $N_{1}=l$. We study which degree $d+3i$ polynomials have $N_{1}>l$. The transformation $\alpha$ increases  $N_{1}$ in one, $\beta$ decreases  $N_{1}$ in one, whereas $\bar{\alpha}, \bar{\beta}$ do not change its value. If we denote by $|W|_{x}$ the number of times the word $W$ contains the transformation $x$, then, after applying the transformations associated  to $W$, we have $N_{1}=s=l+ |W|_{\alpha}-|W|_{\beta}$. The sets $S^{(u)}_{l, r}, u\in \{1, 2\}$ are defined such that $|W|_{\alpha}-|W|_{\beta}=u$ if $i \in S^{(u)}_{l,r}$ and  $|W|_{\alpha}-|W|_{\beta}=0$ if 
    $i \notin S^{(u)}_{l,r}$.  The analysis of the word sequences presented in Eqs. (2.7)-(2.13) gives

      \bigskip\par\noindent
      (a) $\nu=3n+2$, $n \in  {\Bbb{Z}}^{\geq 1}$, $l \in  {\Bbb{Z}}^{\geq 0}$, $r \in \{3p-1, 3p, 3p+1\}$, $p=1, \dots, n-1$: $S^{(1)}_{l, r}=\bigcup_{n-p+1 \leq j \leq n} \{j\}.$
      \bigskip\par\noindent
      (b) $\nu=3n+3$ 
      
      (b1) $n=0$, $l \in  {\Bbb{Z}}^{\geq 0}$, $r=3$:  $S^{(1)}_{l, r}=\{1\}$,              
        
          (b2) $n \in  {\Bbb{Z}}^{\geq 1}$, $l \in  {\Bbb{Z}}^{\geq 0}$, $r \in \{3p+3, 3p+4, 3p+5\}$, $p=0, \dots, n-1$: $S^{(1)}_{l, r}=\bigcup_{1 \leq j \leq p+1} \{j\}$,               
          
             (b3) $n \in  {\Bbb{Z}}^{\geq 1}$, $l \in  {\Bbb{Z}}^{\geq 0}$, $r=3n+3$: $S^{(1)}_{l, r}=\bigcup_{1 \leq j \leq n+1} \{j\}.$              
          
                \bigskip\par\noindent
      (c) $\nu=3n+4$
      
      (c1) $n=0$
      
       \hspace{0.5cm}  (c1.1) $l=0, r=-1$: $S^{(1)}_{l, r}=\{1\}$, 
      
       \hspace{0.5cm}  (c1.2) $l=3k$, $k \in  {\Bbb{Z}}^{\geq 0}$, $r=0$: $S^{(1)}_{l, r}=\{1, 2, 3, 4\}$, 
      
       \hspace{0.5cm}  (c1.3) $l=3k+1$, $k \in  {\Bbb{Z}}^{\geq 0}$, $r=0$:  $S^{(1)}_{l, r}=\{1, 3, 4\}$,
      
      \hspace{0.5cm}  (c1.4) $l=3k+2$, $k \in  {\Bbb{Z}}^{\geq 0}$, $r=-1$:  $S^{(1)}_{l, r}=\{3, 4\}$,
     
       \hspace{0.5cm}   (c1.5) $l=3k+2$, $k \in  {\Bbb{Z}}^{\geq 0}$, $r=0$: $S^{(1)}_{l, r}=\{1, 2\}$ and $S^{(2)}_{l, r}=\{3\}$.
       
          (c2) $n \in  {\Bbb{Z}}^{\geq 1}$  
          
             \hspace{0.5cm}   (c2.1) $l=0, r=-1$: $S^{(1)}_{l, r}=\{n+1\}$,
          
              \hspace{0.5cm}   (c2.2) $l=3k$, $k \in  {\Bbb{Z}}^{\geq 0}$, $r=0, \dots, n-1$: $S^{(1)}_{l, r}=A_{0}\bigcup B_{0}\bigcup C_{0}$, with  $A_{0}=\bigcup_{0 \leq j \leq r}  \{n-r+j+1\}$, $B_{0}=\bigcup_{0 \leq j \leq r+1}  \{2n-r+j+2\}$, $C_{0}=\bigcup_{0 \leq j \leq r}  \{3n-r+j+4\}$, and $S^{(1)}_{l, n}=\bigcup_{1 \leq j \leq 3n+4} \{j\}$,
              
             \hspace{0.5cm}    (c2.3) $l=3k+1$, $k \in  {\Bbb{Z}}^{\geq 0}$, $r=0, \dots, n-1$: $S^{(1)}_{l, r}=A_{1}\bigcup B_{1}\bigcup C_{1}$, with  $A_{1}=\bigcup_{0 \leq j \leq r}  \{n-r+j+1\}$, $B_{1}=\bigcup_{0 \leq j \leq r}  \{2n-r+j+3\}$, $C_{1}=\bigcup_{0 \leq j \leq r}  \{3n-r+j+4\}$, and $S^{(1)}_{l, n}=\bigcup_{1 \leq j \leq 3n+4} \{j\} \setminus \{n+2\}$,
              
               \hspace{0.5cm}    (c2.4) $l=3k+2$, $k \in  {\Bbb{Z}}^{\geq 0}$, $r=-1$:  $S^{(1)}_{l, r} = \{2n+3, 3n+4\}$,  
                
               \hspace{0.5cm}    (c2.5) $l=3k+2$, $k \in  {\Bbb{Z}}^{\geq 0}$, $r=0, \dots, n-2$: $S^{(1)}_{l, r}=A_{2}\bigcup B_{2}\bigcup C_{2}$, with $A_{2}=\bigcup_{0 \leq j \leq r} \{n-r+j+1\}$, $B_{2}=\bigcup_{0 \leq j \leq r + 1} \{2 n-r+j + 2\}$, $C_{2}=\bigcup_{0 \leq j \leq r+1} \{3 n-r+j + 3\}$, 
                
               \hspace{0.5cm}    (c2.6) $l=3k+2$, $k \in  {\Bbb{Z}}^{\geq 0}$, $r=n-1$:  $S^{(1)}_{l, n}=\bigcup_{2 \leq j \leq 3n+4} \{j\} \setminus \{n+2\}$,
                
                \hspace{0.5cm}   (c2.7) $l=3k+2$, $k \in  {\Bbb{Z}}^{\geq 0}$, $r=n$:  $S^{(1)}_{l, r} = \{1, \dots, 2n+2\}$ and  $S^{(2)}_{l, r} =\{2n+3\}$.

\par

\end{document}